\documentclass[11pt]{amsart}
\usepackage {amsmath, amssymb, amscd, mathrsfs, url,hyperref, verbatim, enumerate}
\usepackage[text={6.5in,9in},centering,letterpaper,dvips]{geometry}
\usepackage{color,multirow,latexsym,graphicx,epstopdf,comment,subfig}
\usepackage{varwidth}
\usepackage{pinlabel}
\usepackage[all]{xy}

\newtheorem {theorem}{Theorem}
\newtheorem {lemma}[theorem]{Lemma}
\newtheorem {proposition}[theorem]{Proposition}
\newtheorem {corollary}[theorem]{Corollary}
\newtheorem {conjecture}[theorem]{Conjecture}
\newtheorem {definition}[theorem]{Definition}
\newtheorem {fact}[theorem]{Fact}
\newtheorem {question}[theorem]{Question}

\theoremstyle{remark}

\newtheorem {remark}[theorem]{Remark}

\numberwithin{equation}{section}
\numberwithin{theorem}{section}

\DeclareMathOperator{\rank}{rank}
\def\ep{\varepsilon}
\newcommand{\comments}[1]{}

\DeclareMathOperator{\Z}{\mathbb{Z}}

\newcommand{\Rr}{\mathbb{R}}
\newcommand{\bfalpha}{\boldsymbol{\alpha}}

\newcommand{\bfk}{\mathbf{k}}
\newcommand{\zed}{\mathbb{Z}}
\newcommand{\tsigma}{\widetilde{\sigma}}
\newcommand{\ttau}{\widetilde{\tau}}
\newcommand{\tphi}{\widetilde{\varphi}}
\newcommand{\tgamma}{\widetilde{\gamma}}
\newcommand{\tlambda}{\widetilde{\lambda}}
\newcommand{\talpha}{\widetilde{\alpha}}
\newcommand{\tmu}{\widetilde{\mu}}
\newcommand{\tW}{\widetilde{W}}

\author{Cameron Gordon}
\address{Department of Mathematics\\The University of Texas at Austin\\  Austin TX, 78701\\USA}
\email{gordon@math.utexas.edu}

\author{Tye Lidman}
\address{Department of Mathematics\\North Carolina State University\\ Raleigh NC, 27603\\USA}
\email{tlid@math.ncsu.edu}

\title{Taut foliations, left-orderability, and cyclic branched covers}

\begin{document}
\maketitle

\begin{abstract}
We study the question of when cyclic branched covers of knots admit taut foliations, have left-orderable fundamental groups, and are not L-spaces.  
\end{abstract}

\section{Introduction}\label{sec:intro}

The study of taut foliations has led to a number of important advances in knot theory and three-manifold topology.  Consequently, any connection between taut foliations and other invariants in low-dimensional topology can be quite useful.  One example of this is that taut foliations guarantee the non-triviality of certain analytically-defined invariants associated to three-manifolds, namely various types of Floer homology.  This non-triviality plays a key role in the proofs of Property P \cite{KM} and the Dehn surgery characterization of the unknot \cite{KMOS}.  One such non-triviality statement is that if a rational homology sphere $Y$ admits a co-orientable taut foliation, then $Y$ is not an {\em L-space} (i.e. $\rank \widehat{HF}(Y) > |H_1(Y;\mathbb{Z})|$, where $\widehat{HF}$ denotes the hat-flavor of Heegaard Floer homology) \cite[Theorem 1.4]{OSGenus}.  Consequently, L-spaces form an interesting family of manifolds that have been intensely studied.  Some useful examples of L-spaces to keep in mind for this paper are manifolds with finite fundamental group \cite[Proposition 2.3]{OSLens}.  

Taut foliations are also related to the fundamental group of the underlying three-manifold.  For instance, any loop transverse to the foliation must represent a non-trivial element in $\pi_1$.  More structure can often be found.  Thurston's universal circle construction shows that for a co-orientable taut foliation with hyperbolic leaves, there is an orientation-preserving action of the fundamental group on a circle; for more details see \cite{CD}.  Furthermore, in many cases, one can actually find an orientation-preserving action on the real line, such as when the manifold is an integer homology sphere \cite[Lemma 0.4]{BB} or the foliation is $\mathbb{R}$-covered \cite[Corollary 7.11]{CD}.  Since $\operatorname{Homeo}_+(\mathbb{R})$ is {\em left-orderable} (i.e. admits a left-invariant strict total order), \cite[Theorem 1.1]{BRW} shows that admitting such an action implies that $\pi_1(Y)$ is left-orderable.  

Thus, there is a strong relationship between left-orderability, taut foliations, and L-spaces.  In fact, there are at present no counterexamples known to the possibility that, for an irreducible rational homology sphere $Y$, the following three conditions are equivalent:
\begin{flalign}\label{equiv:lo} &\text{$\pi_1(Y)$ is left-orderable,}&& \end{flalign} $ $ \\ \vspace{-.6in} \\
\begin{flalign}\label{equiv:ctf} &\text{$Y$ admits a co-orientable taut foliation,}&& \end{flalign} $ $ \\ \vspace{-.6in} \\
\begin{flalign}\label{equiv:lspace} &\text{$Y$ is not an L-space.} &&  \end{flalign}
The equivalence of \eqref{equiv:lo} and \eqref{equiv:lspace} was explicitly conjectured in \cite{BGW}.  The equivalence of \eqref{equiv:lo}, \eqref{equiv:ctf}, and \eqref{equiv:lspace}, has been established for many families of three-manifolds, including Seifert manifolds \cite{BGW,BRW, LS}, Sol manifolds \cite{BGW}, and graph manifold homology spheres \cite{BB}.  Also, it is shown in \cite{BC} that \eqref{equiv:lo} and \eqref{equiv:ctf} are equivalent for graph manifolds.  However, the only relation between \eqref{equiv:lo}, \eqref{equiv:ctf}, and \eqref{equiv:lspace} that is unconditionally known is the implication mentioned above: if $Y$ is an L-space, then $Y$ does not admit a co-orientable taut foliation \cite{OSGenus, KR2, B}. 

One reason why the possible equivalence of \eqref{equiv:lo}, \eqref{equiv:ctf}, and \eqref{equiv:lspace} is desirable is that they exhibit different types of behavior.  For example, left-orderability is well-understood under non-zero degree maps, while Floer homology is well-understood under Dehn surgery.  The equivalence of \eqref{equiv:lo}, \eqref{equiv:ctf}, and \eqref{equiv:lspace} would allow one to use these properties in tandem.  We mention that it would follow from the equivalence of \eqref{equiv:lo} and \eqref{equiv:lspace} that a toroidal integer homology sphere can never be an L-space (see \cite[Section 4]{CLW}).  

\begin{definition}
A closed, connected, orientable three-manifold $Y$ is {\em excellent} if $Y$ admits a co-orientable taut foliation and $\pi_1(Y)$ is left-orderable.  Consequently, $Y$ is prime and is not an L-space.  On the other hand, $Y$ is called a {\em total L-space} if all three conditions \eqref{equiv:lo}, \eqref{equiv:ctf}, and \eqref{equiv:lspace} fail.  In other words, a total L-space is an L-space whose fundamental group is not left-orderable.  
\end{definition}

We remark that if $b_1(Y) > 0$ and $Y$ is prime, then $Y$ always admits a co-orientable taut foliation \cite{Gabai1} and has left-orderable fundamental group \cite[Corollary 3.4]{BRW}, and thus is excellent.  Many examples of total L-spaces have come from branched covers of knots, including all two-fold branched covers of non-split alternating links \cite{BGW, OSBranched} and some higher-order branched covers of two-bridge knots \cite{DPT2005,Peters}.  In this paper, we study when cyclic branched covers of certain other families of knots give excellent manifolds or total L-spaces.  We will denote the $n$-fold cyclic branched cover of a knot $K$ by $\Sigma_n(K)$ and will always assume $n \geq 2$.  Note that by the Smith conjecture \cite{BM}, $\Sigma_n(K)$ is simply-connected only if $K$ is trivial.

A key input throughout the paper is that, as mentioned above, the equivalence of \eqref{equiv:lo}, \eqref{equiv:ctf}, and \eqref{equiv:lspace} is known for Seifert fibered manifolds.  This enables us to give a complete analysis of the case of torus knots.  Throughout, we let $T_{p,q}$ denote the $(p,q)$-torus link.  

\begin{theorem}\label{thm:torus} Let $n, p, q \geq 2$ and suppose $p,q$ are relatively prime.  Then, $\Sigma_n(T_{p,q})$ is excellent if and only if its fundamental group is infinite.  Thus, $\Sigma_n(T_{p,q})$ is excellent except in the cases:
\begin{enumerate}[(i)]
\item $\{p,q\} = \{2,3\}$, $2 \leq n \leq 5$,
\item $\{p,q\} = \{2,5\}$, $2 \leq n \leq 3$,
\item $\{p,q\} = \{2,r\}$, $r \geq 7$, $n = 2$,
\item $\{p,q\} = \{3,4\}$, $n = 2$,
\item $\{p,q\} = \{3,5\}$, $n = 2$,
\end{enumerate}
in which case $\Sigma_n(T_{p,q})$ is a total L-space.  
\end{theorem}

Both left-orderability and taut foliations have been studied in the context of gluing manifolds with toral boundary (see for instance \cite{BC, CLW}).  For this reason, we will also study certain families of satellite knots.  While bordered Floer homology is suitable machinery for studying the Heegaard Floer homology of manifolds glued along surfaces \cite{LOT}, we will not focus on its use here.  Let $K'$ be a satellite knot with non-trivial companion $K$; then the complement of $K'$ contains a copy of the exterior $X$ of $K$.  In $\Sigma_n(K')$, the preimage of $X$ consists of copies of some (possibly trivial) cyclic cover of $X$, and the proof of \cite[Theorem 1]{GL} shows that if $K'$ is prime, then the boundary components of this preimage are incompressible in $\Sigma_n(K')$.  

Work of Li and Roberts \cite{LR} provides abundant examples of taut foliations on the exteriors of non-trivial knots in $S^3$ (see Section~\ref{sec:background} for more precise statements).  We will utilize these foliations heavily.  In the case of cables, the remaining piece of $\Sigma_n(K')$ is Seifert fibered, and so we are able to get a nearly complete result in this case.  Let $C_{p,q}(K)$ denote the $(p,q)$-cable of a knot $K$, where $q$ denotes the longitudinal winding.  Assume $q > 1$.        

\begin{theorem}\label{thm:cables}
Let $K$ be a non-trivial knot in $S^3$.  Then, $\Sigma_n(C_{p,q}(K))$ is excellent, except possibly if $n = q = 2$.  
\end{theorem}

It turns out that in the case of $n = q = 2$, we are still able to partially extend Theorem~\ref{thm:cables}.

\begin{theorem}\label{thm:cablefail}
For any non-trivial knot in $S^3$, we have that $\Sigma_2(C_{p,2}(K))$ is not an L-space.  Further, if $K$ is a torus knot or iterated torus knot, then $\Sigma_2(C_{p,2}(K))$ is excellent.   
\end{theorem}

\begin{remark}
In the published version of this article \cite{GLid} and previous arXiv version, it was claimed that $\Sigma_2(C_{p,2}(K))$ was {\em not} excellent if $K$ was the right-handed trefoil and $p > 1$.  The argument there was incorrect.  See Remark~\ref{rmk:correction} below for a more detailed explanation of the mistake.  It is still unknown to the authors if $\Sigma_2(C_{p,2}(K))$ is excellent for arbitrary $K$.  Any counterexample would also provide a counterexample to the conjectural equivalence of \eqref{equiv:lo}, \eqref{equiv:ctf}, and \eqref{equiv:lspace}.
\end{remark}

L-spaces seem to be rare among integer homology spheres: it is conjectured that the only irreducible integer homology sphere L-spaces are $S^3$ and the Poincar\'e homology sphere.  This, and the possibility that \eqref{equiv:lspace} implies \eqref{equiv:lo}, suggests that most integer homology spheres should have left-orderable fundamental group.  Now, if a knot has trivial Alexander polynomial, then all its cyclic branched coverings are integer homology spheres.  A well-known class of knots with trivial Alexander polynomial are the Whitehead doubles.  For the cyclic branched coverings of these knots, we show the following.  

\begin{theorem}\label{thm:whitehead}
Let $K$ be a non-trivial knot and let $Wh(K)$ denote the positive, untwisted Whitehead double of $K$.  Then, $\Sigma_2(Wh(K))$ is an excellent manifold and $\pi_1(\Sigma_n(Wh(K)))$ is left-orderable for all even $n$.  If $\pi_1(S^3_\alpha(K))$ is left-orderable for $\alpha = 1, \frac{1}{2},$ or $\frac{1}{3}$, then $\pi_1(\Sigma_n(Wh(K)))$ is left-orderable for all $n \geq 2$.  
\end{theorem}

In light of Theorems~\ref{thm:cables} and ~\ref{thm:whitehead}, we make the following conjecture.

\begin{conjecture}\label{conj:asymptotics}
Let $K$ be a prime, satellite knot.  For all $n \gg 0$, $\Sigma_n(K)$ is excellent.
\end{conjecture}

The condition that $K$ be prime is necessary, as $\Sigma_n(K)$ is prime (in fact irreducible) if and only if $K$ is prime; this follows from the equivariant sphere theorem \cite{MSY} and the Smith conjecture \cite{BM}.  If $K$ has a summand which is the figure-eight knot, then no cyclic branched cover of $K$ will have left-orderable fundamental group \cite[Theorem 2(c)]{DPT2005}.  

As mentioned, there have been a number of results determining when cyclic branched covers of two-bridge knots have non-left-orderable fundamental group and/or are L-spaces \cite{DPT2005,Hu2013,Peters, Tran2013}.  Let $K_{[p_1,\ldots,p_m]}$ denote the two-bridge knot of the type $\frac{a}{b}$, where $[p_1,\ldots,p_m]$ is the continued fraction expansion for $\frac{a}{b}$.  We follow the convention that $\frac{a}{b} = p_1 + \frac{1}{p_2 + \frac{1}{\ddots + \frac{1}{p_m}}}$.  We give the first examples of cyclic branched covers of hyperbolic two-bridge knots which are excellent manifolds.    

\begin{theorem}\label{thm:twobridgesurgery}
Suppose $k,\ell \geq 1$ and $n$ divides $(2k+1)$ for $n > 1$.  Then $\Sigma_n (K_{[2(2k+1),2\ell+1]})$ is excellent.  
\end{theorem}

Our next result gives a family of cyclic branched covers of two-bridge knots which are total L-spaces.  If $a_i \geq 1$, $1 \leq i \leq m$, then for any $n \geq 2$, the $n$-fold cyclic branched cover of the two-bridge knot $K_{[2a_1,2a_2,\ldots,2a_m]}$ is the two-fold branched cover of an alternating knot or link \cite{MV} and is therefore a total L-space; see \cite{T14}, also \cite{DPT2005,Peters}.  The simplest case where this positivity condition does not hold is the family of two-bridge knots $K_{[2\ell,-2k]}$, $\ell, k \geq 1$.  For these, it is known that $\Sigma_3(K_{[2\ell,-2k]})$ is an L-space \cite{Peters} and has non-left-orderable fundamental group \cite{DPT2005} and is therefore a total L-space.  Recently, Teragaito has shown that $\Sigma_4(K_{[2\ell,-2k]})$ is an L-space \cite{T14}.  We show that $\pi_1(\Sigma_4(K_{[2\ell,-2k]}))$ is not left-orderable, so we have the following.      

\begin{theorem}\label{thm:twobridge}
For $k, \ell \geq 1$, $\Sigma_4(K_{[2\ell,-2k]})$ is a total L-space.  
\end{theorem} 

Finally, Teragaito \cite{Ter} has also shown that the three-fold cyclic branched cover of a three-strand pretzel knot of the form $P (2k+1, 2\ell+1, 2m+1)$, where $k,\ell,m \geq 1$, is an $L$-space.  We show that its fundamental group is not left-orderable.  We thus establish the following.

\begin{theorem}\label{thm:not left}
For $k, \ell, m \geq 1$, $\Sigma_3(P(2k+1,2\ell+1,2m+1))$ is a total L-space.  
\end{theorem}

With this, we can improve \cite[Corollary 1.2]{Ter} to the following.  
\begin{corollary}
Let $K$ be a genus one, alternating knot.  Then $\Sigma_3(K)$ is a total L-space.
\end{corollary}
\begin{proof}
We repeat the argument of \cite[Corollary 1.2]{Ter}, where it is shown that $\Sigma_3(K)$ is an L-space.  We are interested in showing $\pi_1(\Sigma_3(K))$ is not left-orderable.  If $K$ is a genus one, alternating knot, then $K$ is either a genus one, two-bridge bridge knot or, up to mirroring, a pretzel knot $P(2k+1,2\ell+1,2m+1)$ with $k,\ell,m \geq 1$ \cite[Lemma 3.1]{BZ} (independently \cite{Patton}).  The former case follows from \cite{DPT2005}, while the latter is Theorem~\ref{thm:not left}.     
\end{proof}

\noindent {\bf Organization:} In Section~\ref{sec:background}, we review the relevant definitions and results that we will invoke throughout the paper.  In Section~\ref{sec:torus} we prove Theorem~\ref{thm:torus}.  In Section~\ref{sec:cables} we prove Theorems~\ref{thm:cables} and ~\ref{thm:cablefail}.  In Sections~\ref{sec:whitehead}, ~\ref{sec:twobridgesurgery}, ~\ref{sec:twobridge}, and ~\ref{sec:pretzels} we prove Theorems~\ref{thm:whitehead}, ~\ref{thm:twobridgesurgery}, ~\ref{thm:twobridge}, and ~\ref{thm:not left} respectively.  Finally, we briefly discuss two other families of cyclic branched covers in Section~\ref{sec:misc}. \\

\noindent {\bf Acknowledgments:}  We would like to thank Steve Boyer, Nathan Dunfield, Adam Levine, and Rachel Roberts for helpful discussions.  We are particularly grateful to Nathan Dunfield for providing the information described in Remark~\ref{rmk:nathan}\eqref{item:nathan}.  The first author was partially supported by NSF Grant DMS-1309021.  The second author was partially supported by NSF Grant DMS-0636643.

\section{Background}\label{sec:background}

In this section, we will collect results from the literature and other technical lemmas which will be used throughout the paper.  All three-manifolds are assumed to be connected, compact, and orientable unless specified otherwise.  For a knot $K$ in $S^3$, we use the notation $-K$ to denote the mirror of $K$.

\subsection{Taut foliations and left-orders}\label{sec:tautlo}
We begin with the relevant background on taut foliations and left-orders.  We will discuss these simultaneously so we can draw numerous parallels.  

A {\em taut foliation} $\mathcal{F}$ on a three-manifold is a foliation by surfaces such that for each leaf $F \in \mathcal{F}$, there exists a curve $\gamma$ which is transverse to $\mathcal{F}$ and intersects the leaf $F$.  
Manifolds with taut foliations are prime and have infinite fundamental group (see, for instance, \cite{Calegari2007}).  

A group $G$ is {\em left-orderable} if there exists a left-invariant, strict total order on $G$.  We use the convention that the trivial group is {\em not} left-orderable.  Examples of left-orderable groups include $\mathbb{Z}$, braid groups \cite{Dehornoy}, and $\operatorname{Homeo}_+(\Rr)$, while any group with torsion (e.g. a finite group) is not left-orderable.  One reason why the orderability of three-manifold groups is particularly well-suited for study is the following: 

\begin{theorem}[Boyer-Rolfsen-Wiest, Theorem 1.1 in \cite{BRW}]\label{thm:brw}
Let $Y$ be a compact, connected, irreducible, $P^2$-irreducible three-manifold.  Then, if there exists a non-zero homomorphism from $\pi_1(Y)$ to a left-orderable group, then $\pi_1(Y)$ is left-orderable.  In particular, if there exists a non-zero degree map from $Y$ to $Y'$ and $\pi_1(Y')$ is left-orderable, then so is $\pi_1(Y)$.  
\end{theorem}

Note that for closed, orientable manifolds other than $\mathbb{R}P^3$, irreducibility implies $P^2$-irreducibility.  Since there are no non-trivial homomorphisms from $\pi_1(\mathbb{R}P^3)$ to a left-orderable group, this case will not be a concern.  Further, observe that $\pi_1(S^2 \times S^1)$ is left-orderable, so we may further replace irreducible with prime.  Theorem~\ref{thm:brw} implies that a prime three-manifold with $b_1 > 0$ will have left-orderable fundamental group.  Similarly, Gabai showed that a prime three-manifold with $b_1 > 0$ always has a co-orientable taut foliation \cite{Gabai1}.  Recall that we say a closed three-manifold is {\em excellent} if it admits a co-orientable taut foliation and has left-orderable fundamental group.  Any prime three-manifold with $b_1 > 0$ is thus automatically excellent.  

\begin{remark}
As we will repeatedly invoke Theorem~\ref{thm:brw}, we want to ensure that the cyclic branched covers we work with are prime.  All knots that we will take cyclic branched covers of in this paper (torus knots, cables, Whitehead doubles, two-bridge knots, and pretzel knots) will be prime.  Therefore, all such cyclic branched covers will be prime; we will not mention this point again.
\end{remark}  

We also would like notions of left-orders and taut foliations with appropriate boundary conditions.  We will only consider rational slopes on boundary tori in this paper.  

\begin{definition}
Let $M$ be a compact three-manifold with boundary a disjoint union of tori, $T_1,\ldots,T_n$, for some $n \geq 1$.  Let $\alpha_i$ be a slope on $T_i$ for each $1 \leq i \leq n$.  The multislope $\bfalpha = (\alpha_1,\ldots,\alpha_n)$ is called a {\em CTF multislope} if $M$ has a co-orientable taut foliation which meets $T_i$ transversely in circles of slope $\alpha_i$ for $1 \leq i \leq n$.  If $M$ is the exterior of a knot $K$ in $S^3$, we say that $\alpha$ is a {\em CTF slope for $K$} if $\alpha$ is a CTF slope for $M$.  
\end{definition}

Observe that if $\bfalpha$ is a CTF multislope on $M$, the manifold $M(\bfalpha)$ admits a co-orientable taut foliation.  Recently, Li and Roberts have shown that sufficiently small slopes on knots in $S^3$ are always CTF slopes.  

\begin{theorem}[Li-Roberts, Theorem 1.1 in \cite{LR}]\label{thm:liroberts}
Let $K$ be a non-trivial knot in $S^3$.  Then, there exists an interval $(-a,b)$ with $a, b > 0$ such that if $\alpha \in (-a,b)$ then $\alpha$ is a CTF slope for $K$.     
\end{theorem}

In fact, they conjecture the existence of a universal interval for all knots.  
\begin{conjecture}[Li-Roberts, Conjecture 1.9 in \cite{LR}]\label{conj:lr}
Let $K$ be a non-trivial knot in $S^3$.  If $\alpha \in (-1,1)$, then $\alpha$ is a CTF slope for $K$.  
\end{conjecture}

We point out that the Li-Roberts Conjecture is known for many families of knots, including hyperbolic fibered knots \cite{Roberts2001} and non-special alternating knots \cite{Roberts1995}.  

In analogy with CTF multislopes, we have the following definition for left-orderability, and consequently excellence.

\begin{definition}
We say that $\bfalpha$ is an {\em LO multislope} if $\pi_1(M(\bfalpha))$ is left-orderable.  We say that $\bfalpha$ is {\em excellent} if it is both an LO multislope and a CTF multislope.  If $M$ is the exterior of a knot $K$ in $S^3$, we will say that $\alpha$ is an {\em LO slope for $K$} (respectively {\em excellent slope for $K$}) if it is an LO slope (respectively excellent slope) for $M$.  
\end{definition}

Observe that if $\bfalpha$ is an excellent multislope, the manifold $M(\bfalpha)$ is excellent.  

There is a strong relationship between taut foliations and left-orderability for homology spheres.  Note that on an integer homology sphere, any foliation is automatically co-orientable.  In \cite{CD}, it is shown that an atoroidal homology sphere with a taut foliation has left-orderable fundamental group.  
This was then extended to all integer homology spheres.  
\begin{lemma}[Boileau-Boyer, Lemma 0.4 in \cite{BB}]\label{lem:toroidalzs3}
Suppose that $Y$ is an integer homology sphere admitting a taut foliation.  Then, $\pi_1(Y)$ is left-orderable.  
\end{lemma}
Therefore, if an integer homology sphere $Y$ admits a taut foliation, $Y$ is automatically excellent.  Using this observation and Theorem~\ref{thm:liroberts}, we can construct numerous excellent slopes on exteriors of knots in $S^3$.  

\begin{lemma}\label{lem:1/n}
Let $K$ be a non-trivial knot in $S^3$.  Then, there exists $k_0$ such that if $|k| > k_0$, $\frac{1}{k}$ is an excellent slope for $K$.  
\end{lemma}
\begin{proof}
By Theorem~\ref{thm:liroberts}, there exists $k_0$ such that if $|k| > k_0$, then $\frac{1}{k}$ is a CTF slope for $K$.  Since $H_1(S^3_{\frac{1}{k}
}(K)) = 0$, Lemma~\ref{lem:toroidalzs3} implies that $\frac{1}{k}$ is an LO slope for $K$.  Thus, $\frac{1}{k}$ is excellent for $|k| > k_0$.  
\end{proof}

\subsection{Toroidal manifolds}
We are interested in the more general question of when we can glue pieces with toral boundary together to obtain manifolds with taut foliations and/or left-orderable fundamental group.  Let $M$ be a compact three-manifold with boundary a disjoint union of tori, $S_1,\ldots,S_m$, for $m \geq 0$.  Further, let $\mathcal{J}$ be a disjoint union of incompressible, separating tori $T_1,\ldots, T_n$ in $M$ which are not boundary parallel.  Let $X_1,\ldots,X_{n+1}$ be the components of $M$ cut along $\mathcal{J}$.  Choosing a slope on each component of $\partial M \cup \mathcal{J}$ defines a multislope $\bfalpha_i$ on the boundary of $X_i$ for each $1 \leq i \leq n+1$.    

\begin{lemma}\label{lem:excellentgluing}
Let $M$ be as above.  If $\bfalpha_i$ is a CTF multislope for each $X_i$, then $\bfalpha|_{\partial M}$ is a CTF multislope for $M$.   If in addition, $\bfalpha_i$ is an LO multislope for each $X_i$, then $\bfalpha|_{\partial M}$ is an LO multislope for $M$.  Consequently, if $\bfalpha_i$ is excellent for each $X_i$, then $\bfalpha|_{\partial M}$ is excellent for $M$.  
\end{lemma}

Lemma~\ref{lem:excellentgluing} can be rephrased in a simple way for the case that $M$ is closed.  In this case, it says that if one glues manifolds with toral boundary such that CTF multislopes (respectively excellent multislopes) are identified, the resulting manifold still has a co-orientable taut foliation (respectively is excellent).  The first claim in Lemma~\ref{lem:excellentgluing} is in fact trivial.  One can simply glue the foliations on each $X_i$ together.  We are thus interested in showing that $\bfalpha|_{\partial M}$ is an LO slope on $M$.  This will be a generalization of the following theorem.   

\begin{theorem}[Clay-Lidman-Watson, Theorem 2.7 in \cite{CLW}]\label{thm:jsjlo}
Let $X_1$ and $X_2$ be compact, oriented, connected three-manifolds with incompressible torus boundary.  Suppose $\alpha_1$ and $\alpha_2$ are LO slopes on $X_1$ and $X_2$ respectively, and $f: \partial X_1 \to \partial X_2$ is an orientation-reversing homeomorphism such that $f(\alpha_1) = \alpha_2$.  If $Y = X_1 \cup_f X_2$ is irreducible, then $Y$ has left-orderable fundamental group.  
\end{theorem}

The argument we use to prove Lemma~\ref{lem:excellentgluing} will essentially repeat that of Theorem~\ref{thm:jsjlo}.  We refer the reader to the proof of that theorem for more details.  
\begin{proof}[Proof of Lemma~\ref{lem:excellentgluing}]
Let $T_1,\ldots,T_n$ be the components of $\mathcal{J}$ and let $S_1,\ldots,S_m$ be the components of $\partial M$, as above.  We abuse notation and write $M(\bfalpha)$ for $M(\bfalpha|_{\partial M})$.  We would like to show that $\pi_1(M(\bfalpha))$ is left-orderable.  We prove the result by induction on $n$.  First, let $n = 0$.  In this case, we have that $(\alpha_{1},\ldots,\alpha_{m})$ is an LO slope by assumption.  Now suppose that the result holds for a non-negative integer $n$.  We will show the result holds for $n+1$.  Recall that each $T_i \in \mathcal{J}$ is separating.  Let $M$ cut along $T_1$ have components $X'$ and $X''$ and let $\bfalpha'$ and $\bfalpha''$ be the induced multislopes on $\partial X'$ and $\partial X''$ respectively.  

Let $\alpha_1$ denote the element of $\bfalpha$ which is a slope on $T_1$.  By induction, we have that $\bfalpha'$ and $\bfalpha''$ are LO slopes on $X'$ and $X''$ respectively.  We study the normal closure of $\bfalpha'$, $\langle \langle \bfalpha' \rangle \rangle$, in $\pi_1(X')$.  We consider $\langle \langle \bfalpha' \rangle \rangle \cap \pi_1(T_1)$.  This group is either $\langle \alpha_1 \rangle$, $\pi_1(T_1)$, or $\mathbb{Z} \oplus k\mathbb{Z}$ for some $k \geq 2$.  Observe that the third possibility cannot occur, since in this case $\pi_1(X'(\bfalpha'))$ would have torsion, contradicting the assumption that $\pi_1(X'(\bfalpha'))$ is left-orderable.  A similar discussion applies for $\bfalpha''$ on $X''$.  

We first consider the case that $\langle \langle \bfalpha' \rangle \rangle \cap \pi_1(T_1) = \pi_1(T_1)$ in $\pi_1(X')$.  Observe that in this case the quotient maps $\pi_1(X') \to \pi_1(X')/\langle \langle \bfalpha' \rangle \rangle$ and $\pi_1(X'') \to 1$ agree on the subgroup $\pi_1(T_1)$.  Since 
\[
\pi_1(M) = \pi_1(X') *_{\pi_1(T_1)} \pi_1(X''),
\]
we obtain a quotient $\pi_1(M) \to \pi_1(X'(\bfalpha'))$.  Since $\langle \langle \bfalpha|_{\partial M} \rangle \rangle$ is in the kernel, we have an induced quotient $\pi_1(M(\bfalpha)) \to \pi_1(X'(\bfalpha'))$.  Because $M(\bfalpha)$ admits a co-orientable taut foliation by assumption, we have that $M(\bfalpha)$ is prime.  Since $\pi_1(X'(\bfalpha'))$ is left-orderable by our induction hypothesis, Theorem~\ref{thm:brw} shows that $\pi_1(M(\bfalpha))$ is left-orderable.  

The case that $\langle \langle \bfalpha'' \rangle \rangle \cap \pi_1(T_1) = \pi_1(T_1)$ in $\pi_1(X'')$ is similar.  Therefore, we assume that $\langle \langle \bfalpha' \rangle \rangle \cap \pi_1(T_1) = \langle \alpha_1 \rangle$ in $\pi_1(X')$ and $\langle \langle \bfalpha'' \rangle \rangle \cap \pi_1(T_1) = \langle \alpha_1 \rangle$ in $\pi_1(X'')$.  Thus, $\pi_1(T_1)$ has cyclic image in $\pi_1(X'(\bfalpha'))$ and $\pi_1(X''(\bfalpha''))$.  We have an induced quotient 
\[
\pi_1(M) \to \pi_1(X'(\bfalpha')) *_{\mathbb{Z}} \pi_1(X''(\bfalpha'')).
\]  
Observe that $\langle \langle \bfalpha|_{\partial M} \rangle \rangle \subset \pi_1(M)$ is contained in the kernel of this quotient map.  Therefore, we obtain a quotient
\[
\pi_1(M(\bfalpha)) \to \pi_1(X'(\bfalpha')) *_{\mathbb{Z}} \pi_1(X''(\bfalpha'')).
\]  
Each of $\pi_1(X'(\bfalpha'))$ and $\pi_1(X''(\bfalpha''))$ is left-orderable by the induction hypothesis.  Therefore, the group $\pi_1(X'(\bfalpha')) *_{\mathbb{Z}} \pi_1(X''(\bfalpha''))$ is left-orderable by \cite[Corollary 5.3]{BG}.  By assumption $M(\bfalpha)$ has a co-orientable taut foliation and therefore is prime.  By Theorem~\ref{thm:brw} we have that $\pi_1(M(\bfalpha))$ is left-orderable.  This completes the proof.                       
\end{proof}

\subsection{Branched covers}\label{subsec:branchedcovers}
Let $K$ be a nullhomologous knot in a rational homology sphere $Y$ with exterior $X$.  Then the generator of $H_2(X, \partial X) \cong H^1(X) \cong \mathbb{Z}$ is represented by a properly embedded orientable surface $F$ with a single boundary component isotopic to $K$ in a neighborhood of $K$.  Let $\mu$ be a meridian of $K$.  The map from $\mathbb{Z}$ to $H_1(X)$ which sends a generator to $\mu$ has a unique splitting $H_1(X) \to \mathbb{Z}$ defined by $[\gamma] \mapsto \gamma \cdot F$.  The corresponding maps $\pi_1(X) \to \mathbb{Z}$ and $\pi_1(X) \to \mathbb{Z}/n$ define canonical infinite and $n$-fold cyclic coverings $X_\infty \to X$ and $X_n \to X$ respectively.  Note that $\mu^n$ lifts to a simple loop $\mu_n \subset \partial X_n$.  The {\em $n$-fold cyclic branched covering of $K$} is then defined to be $\Sigma_n(K) = X_n(\mu_n)$.  The obvious branched covering projection $p_n: \Sigma_n(K) \to Y$ factors through a branched cover $p_{n,m}: \Sigma_n(K) \to \Sigma_m(K)$ for any $m$ which divides $n$; it is clear we have $p_n = p_m \circ p_{n,m}$.  Observe that both $p_n$ and $p_{n,m}$ are non-zero degree maps.  However, in the proof of Theorem~\ref{thm:twobridgesurgery}, we will need something stronger.  
\begin{lemma}\label{lem:pi1surjective}
The induced map $(p_{n,m})_*: \pi_1(\Sigma_n(K)) \to \pi_1(\Sigma_m(K))$ is surjective.  
\end{lemma}
\begin{proof}
We have inclusions $\pi_1(X_\infty) \subset \pi_1(X_n) \subset \pi_1(X_m) \subset \pi_1(X)$ induced by the covering maps.  If $g \in \pi_1(X)$, then $g = h\mu^k$ for some $k \in \mathbb{Z}$ and some $h \in \pi_1(X_\infty)$.  Let $q_m : \pi_1(X_m) \to \pi_1(\Sigma_m(K))$ be induced by the quotient map $X_m \to \Sigma_m(K)$.  Note that $\ker q_m = \langle \langle \mu^m \rangle \rangle$.  Similarly, we have the map $q_n: \pi_1(X_n) \to \pi_1(\Sigma_n(K))$.  

Let $x \in \pi_1(\Sigma_m(K))$.  Then $x = q_m(g)$ for some $g \in \pi_1(X_m)$.  Writing $g = h\mu^k$ as above, $g \in \pi_1(X_m)$ implies $k \equiv 0 \pmod{m}$.  Therefore, $q_m(g) = q_m(h) \in \pi_1(\Sigma_m(K))$.  Since $q_n(h) \in \pi_1(\Sigma_n(K))$, and $(p_{n,m})_*(q_n(h)) = q_m(h)$, the result follows.
\end{proof}

Let $X$ be the exterior of a knot $K$ in $S^3$ and let $\lambda, \mu$ be a longitude-meridian pair on $\partial X$.  Under the $n$-fold cyclic covering $X_n \to X$, we have that $\lambda$ lifts to a simple closed curve $\widetilde{\lambda}$ and $\mu^n$ lifts to a simple closed curve $\widetilde{\mu}$.  We use $\widetilde \lambda, \widetilde \mu$ as a basis for slopes on $\partial X_n$.  Note that the inverse image of the slope $\mu + k \lambda$ is a single circle of slope $\widetilde{\mu} + nk \widetilde \lambda$.  

\begin{lemma}\label{lem:cover1/n}
Let $K$ be a non-trivial knot in $S^3$ and let $M_n$ denote the $n$-fold cyclic cover of the exterior of $K$.  Then, there exists an integer $k_0$ such that if $|k| > k_0$, then $\widetilde{\mu} + nk \widetilde{\lambda}$ is an excellent slope for $M_n$.  
\end{lemma} 
\begin{proof}
By Theorem~\ref{thm:liroberts}, we have $k_0$ such that for $|k| > k_0$, $\mu + k \lambda$ is a CTF slope for $K$.  Let $\mathcal{F}$ be such a taut foliation on the exterior of $K$.  Then the preimage of $\mathcal{F}$ under the covering projection from $M_n$ to $M$ is a taut foliation on $M_n$ which intersects the boundary in simple closed curves of slope $\widetilde \mu +  nk \widetilde{\lambda}$.  Therefore, $\widetilde \mu + nk \widetilde{\lambda}$ is a CTF slope on $M_n$.  

Observe that $M_n(\widetilde \mu + nk \widetilde \lambda)$ is an $n$-fold cyclic branched cover of $M(\mu + k \lambda)$.  Therefore, there exists a non-zero degree map from $M_n(\widetilde \mu + nk \widetilde \lambda)$ to $M(\mu + k \lambda)$.  By Lemma~\ref{lem:1/n}, we have that $\pi_1(M(\mu + k \lambda))$ is left-orderable for $|k| > k_0$.  Since $\widetilde \mu + nk \widetilde \lambda$ is a CTF slope on $M_n$, we must have that $M_n(\widetilde \mu + nk \widetilde \lambda)$ is irreducible.  Theorem~\ref{thm:brw} implies that  $\pi_1(M_n(\widetilde \mu + nk \widetilde \lambda))$ is left-orderable.  Thus, $\widetilde \mu + nk \widetilde \lambda$ is excellent for $|k| > k_0$.
\end{proof}

\subsection{Seifert manifolds}

A three-manifold $Y$ is {\em Seifert fibered} if it admits a foliation by circles.  Quotienting $Y$ by the obvious $S^1$-action yields an orbifold, called the {\em base orbifold}.  A rational homology sphere always has base orbifold $S^2$ or $\mathbb{R}P^2$ and an integer homology sphere always has base orbifold $S^2$.  Combining \cite[Proposition 5]{BGW} and \cite[Theorem 1.3]{BRW}, we have that for rational homology spheres, if the base orbifold is $\mathbb{R}P^2$, then $Y$ is a total L-space.  

Whenever we work explicitly with the Seifert invariants of a manifold, we will restrict to the case of base orbifold $S^2$.  We will write our Seifert invariants as $M(\frac{\beta_1}{\alpha_1},\ldots,\frac{\beta_m}{\alpha_m})$.  Here, we require that $\alpha_i, \beta_i$ are relatively prime, and if $\beta_i = 0$ then $\alpha_i = 1$. We will often use {\em normalized} Seifert invariants, as in \cite{EHN}, which take the form $M(b,\frac{\beta_1}{\alpha_1},\ldots,\frac{\beta_n}{\alpha_n})$ where $b \in \mathbb{Z}$ ($b$ is the {\em Euler number}) and $0 < \beta_i < \alpha_i$ for each $i$.  We will usually point out explicitly when we are using normalized Seifert invariants.

We now make our conventions more explicit in terms of orientations.  Suppose that $r,s$ are relatively prime positive integers.  If $0 < r' < r$, $0 < s' < s$, and $b$ are such that $brs + r's + s'r = -1$, then the Seifert structure $M(b,\frac{r'}{r},\frac{s'}{s})$ on $S^3$ has regular fibers given by the {\em positive} $(r,s)$-torus knots.  On the other hand, the Seifert structure $M(-b,-\frac{r'}{r},-\frac{s'}{s})$ on $S^3$ has regular fibers given by the $(-r,s)$-torus knots.  More generally, we have 
\begin{equation}\label{eqn:seifertorientation}
-M\left(b,\frac{\beta_1}{\alpha_1},\ldots,\frac{\beta_n}{\alpha_n}\right) = M\left(-b,-\frac{\beta_1}{\alpha_1},\ldots,-\frac{\beta_n}{\alpha_n}\right) = M\left(-n-b, \frac{\alpha_1 - \beta_1}{\alpha_1},\ldots, \frac{\alpha_n - \beta_n}{\alpha_n}\right).
\end{equation}   
Note that if $Y$ has normalized Seifert invariants  $M(b,\frac{\beta_1}{\alpha_1},\ldots,\frac{\beta_n}{\alpha_n})$, then $M(-n-b, \frac{\alpha_1 - \beta_1}{\alpha_1},\ldots, \frac{\alpha_n - \beta_n}{\alpha_n})$ give the normalized Seifert invariants of $-Y$.    

We will repeatedly use that \eqref{equiv:lo}, \eqref{equiv:ctf}, and \eqref{equiv:lspace} are equivalent for Seifert manifolds.  In fact, a stronger result holds.  Recall that a {\em horizontal foliation} of a Seifert manifold is a codimension one foliation which is transverse to the fibers.  By definition, such foliations are taut.  Further, for orientable Seifert fibered manifolds, a horizontal foliation is co-orientable if and only if the base orbifold is orientable \cite[Lemma 5.5]{BRW}.  For the following, we specialize the result to the case of closed, orientable manifolds.      

\begin{theorem}[\cite{BGW},\cite{BRW},\cite{LS}]\label{thm:sfequiv}
Let $Y$ be a closed, orientable Seifert fibered space.  Then, the following are equivalent:
\begin{enumerate}
\item $Y$ admits a co-orientable taut foliation, 
\item either $b_1(Y) = 0$, $Y$ has base orbifold $S^2$, and admits a horizontal foliation, or $b_1(Y) > 0$, 
\item $Y$ is not an L-space, 
\item $\pi_1(Y)$ is left-orderable.  
\end{enumerate}
\end{theorem}

Note that if $b_1(Y) = 0$ and any of the conditions of Theorem~\ref{thm:sfequiv} are satisfied, then $Y$ has base orbifold $S^2$.  A key part of the proof of Theorem~\ref{thm:sfequiv} comes from the classification of Seifert manifolds admitting horizontal foliations, due to Eisenbud-Hirsch-Neumann, Jankins-Neumann, and Naimi.  We state only the case for closed, orientable manifolds with base orbifold $S^2$.  

\begin{theorem}[\cite{EHN, JN, Naimi1994}]\label{thm:brwcondition}
Let $Y$ be a Seifert fibered space with base orbifold $S^2$ and normalized Seifert invariants $M(b,\frac{\beta_1}{\alpha_1}, \ldots, \frac{\beta_n}{\alpha_n})$ where $n \geq 3$ (and thus $0 < \beta_j < \alpha_j$).  Then, $Y$ admits a horizontal foliation if and only if one of the following holds: 
\begin{enumerate}
\item \label{condition:fibers} $-(n-2) \leq b \leq -2$,
\item \label{condition:-1} $b = -1$ and there are integers $0 < a < m$ such that after some permutation of the $\frac{\beta_j}{\alpha_j}$, we have that $\frac{\beta_1}{\alpha_1} < \frac{a}{m}$, $\frac{\beta_2}{\alpha_2} < \frac{m-a}{m}$, and $\frac{\beta_j}{\alpha_j} < \frac{1}{m}$ for $3 \leq j \leq n$, 
\item \label{condition:-(n-1)} $b = -(n-1)$ and \eqref{condition:-1} holds for $-Y = M(-1, \frac{\alpha_1 - \beta_1}{\alpha_1}, \ldots, \frac{\alpha_n-\beta_n}{\alpha_n})$.  
\end{enumerate}
\end{theorem}

\begin{remark}
In \cite{BRW}, the condition that $a$ and $m$ be relatively prime is also included.  However, this condition is easily shown to be redundant.
\end{remark}

Given an explicit Seifert manifold, Theorem~\ref{thm:brwcondition} provides a concrete means of determining whether it is excellent or if it is a total L-space.  Note that the horizontal foliations guaranteed by Theorem~\ref{thm:brwcondition} are necessarily co-orientable; further, in Theorem~\ref{thm:brwcondition}, there are no assumptions on the first Betti number.  
Using this, we can easily understand the relationship between foliations on Seifert manifolds with torus boundary and foliations on their Dehn fillings.  As discussed above, if $\bfalpha$ is an excellent multislope on $M$, then $M(\bfalpha)$ is an excellent manifold.  The converse holds in many cases for Seifert manifolds.  
\begin{lemma}\label{lem:seifertboundary}
Let $M$ be an orientable Seifert fibered manifold with boundary tori $T_1, \ldots, T_n$ and let $\bfalpha$ be a multislope on $M$.  If $M(\bfalpha)$ is Seifert fibered over $S^2$ and admits a horizontal foliation, then $\bfalpha$ is an excellent multislope on $M$ and $M(\bfalpha)$ is an excellent manifold. 
\end{lemma}
\begin{proof}
Our assumptions guarantee that $M(\bfalpha)$ is an excellent manifold by Theorem~\ref{thm:sfequiv}.  By definition of excellence, $\bfalpha$ is an LO slope on $M$.  It remains to show that $\bfalpha$ is a CTF slope for $M$.  Consider a horizontal foliation $\mathcal{F}_0$ on $M(\bfalpha)$.  Since the base orbifold is $S^2$, $\mathcal{F}_0$ is co-orientable.  Because the cores of the Dehn fillings are fibers in the Seifert fibration, they are necessarily transverse to the leaves of the horizontal foliation.  The desired co-orientable foliation on $M$ intersecting each $T_i$ in simple closed curves of slope $\alpha_i$ is simply the restriction of $\mathcal{F}_0$ to $M$.         
\end{proof}

\subsection{Surgery on torus links}
Many of the cyclic branched covers  that we will encounter later on will contain the exterior of a torus link.  Therefore, we are interested in which multislopes on torus link exteriors are excellent.  


\begin{proposition}\label{prop:negtorus}
Let $k_1,\ldots,k_d$ be integers which are at least 2 and let $r,s$ be relatively prime, positive integers.  Here, we allow $d = 1$ only when $r$ and $s$ are both at least 2 and we allow $r = s = 1$ only when $d \geq 3$.  Then, $-\bfk = (-k_1,\ldots,-k_d)$ is an excellent multislope on the exterior of the torus link $T_{dr,ds}$.  
\end{proposition}
\begin{proof}
By Lemma~\ref{lem:seifertboundary}, it suffices to show that $S^3_{-\bfk}(T_{dr,ds})$ is a Seifert fibered space with base orbifold $S^2$ and admits a horizontal foliation.   
We do this by explicitly computing the Seifert invariants of $(-k_1,\ldots,-k_d)$-surgery on $T_{dr,ds}$ and applying Theorem~\ref{thm:brwcondition}.  To do this, we first would like to find a Seifert structure on $S^3$ such that $T_{dr,ds}$ consists of a collection of regular fibers, each isotopic to $T_{r,s}$.  We consider three cases.  The first case is that $r,s \geq 2$.  

Following \cite{NR}, we consider the Seifert fibration, $M(0,\frac{\beta_1}{r}, \frac{\beta_2}{s})$, of $S^3$, where $\beta_1s + \beta_2r = -1$.  We choose $\beta_2$ so that $0 < \beta_2 < s$.  Then $\beta'_1 = \beta_1 + r$ satisfies $0 < \beta'_1 < r$, and the normalized Seifert invariants are now $M(-1,\frac{\beta'_1}{r},\frac{\beta_2}{s})$.  The torus link $T_{dr,ds}$ consists of $d$ parallel regular fibers in the Seifert fibration, $K_1,\ldots,K_d$.  Let $\mu_i$ and $\lambda_i$ denote the meridian and longitude of $K_i$ respectively, and let $\varphi_i$ denote the fiber slope on $\partial N(K_i)$.  

We study the result of Dehn filling the boundary of the exterior of $T_{dr,ds}$ by the slopes $\gamma_i = a_i\mu_i + b_i\varphi_i$.  By our orientation conventions, $-\mu_i,\varphi_i$ gives an oriented section-fiber basis for each boundary torus of the link exterior.   By construction, as long as $a_i \neq 0$ for all $i$, the filled manifold has Seifert invariants  
\[
M\left (-1,\frac{\beta'_1}{r}, \frac{\beta_2}{s}, -\frac{b_1}{a_1},\ldots, -\frac{b_d}{a_d}\right) .  
\]          
Recall that we are interested in $(-k_1,\ldots,-k_d)$-surgery.  This corresponds to filling the boundary tori by slopes $-k_i \mu_i + \lambda_i$.  Note that the longitude $\lambda_i$ satisfies $\varphi_i = \lambda_i + rs \mu_i$.  This can be seen, for instance, by noting that the linking number of the $(r,s)$-torus knot with a parallel regular fiber is $rs$.  Therefore, $-k_i \mu_i + \lambda_i = (-k_i - rs) \mu_i + \varphi_i$.  We conclude, if $k_1,\ldots,k_d > 0$, that the normalized Seifert invariants for $(-k_1,\ldots,-k_d)$-surgery are given by  
\[
S^3_{-\bfk}(T_{dr,ds}) = M \left(-1, \frac{\beta'_1}{r}, \frac{\beta_2}{s}, \frac{1}{k_1+rs},\ldots, \frac{1}{k_d+rs} \right).
\]
Let $m = rs+1$ and $a = \beta_2 r + 1$.  Then $\frac{\beta_2}{s} < \frac{a}{m}$, and 
\[
\frac{\beta'_1}{r} = \frac{rs - 1 - \beta_2 r}{rs} = \frac{m - a -1}{m-1} < \frac{m-a}{m}.  
\]
Also, if $k_i \geq 2$, then $\frac{1}{k_i+rs} < \frac{1}{m}$.  Thus, Condition \eqref{condition:-1} of Theorem~\ref{thm:brwcondition} holds, and we can conclude that $S^3_{-\bfk}(T_{dr,ds})$ admits a horizontal foliation.   This completes the proof for the case that both $r$ and $s$ are at least 2.  

The next case is when exactly one of $r$ or $s$ is 1.  Without loss of generality, $r = 1$ and thus, by assumption, $s \geq 2$ and $d \geq 2$.  We consider the Seifert structure $M(-1,\frac{s-1}{s})$ on $S^3$.  In this case $T_{d,ds}$ is given by the union of $d$ parallel regular fibers, $K_1,\ldots,K_d$ (each fiber is an unknot).  By arguments similar to those in the previous case, we compute the normalized Seifert invariants for $S^3_{-\bfk}(T_{d,ds})$ to be
\[
S^3_{-\bfk}(T_{d,ds}) = M\left(-1, \frac{s-1}{s}, \frac{1}{k_1 + s}, \ldots, \frac{1}{k_d + s}\right).  
\]
Since $d \geq 2$, $S^3_{-\bfk}(T_{d,ds})$ has at least three singular fibers.  Because $k_1,\ldots,k_d \geq 2$, we see that Condition~\eqref{condition:-1} of Theorem~\ref{thm:brwcondition} is satisfied by choosing $a = 1$ and $m = s+1$.  

The final case to consider is $r = s = 1$ and $d \geq 3$.  This case is similar to the previous one.  We have that $T_{d,d}$ is a collection of $d$ parallel regular fibers in the Seifert structure $M(-1)$ on $S^3$.  Therefore, we compute the Seifert invariants for $S^3_{-\bfk}(T_{d,d})$ to be
\[
S^3_{-\bfk}(T_{d,d}) = M\left(-1, \frac{1}{k_1 +1}, \ldots, \frac{1}{k_d + 1} \right).  
\] 
Since $d \geq 3$, $S^3_{-\bfk}(T_{d,d})$ has at least three singular fibers.  We again see that because $k_1, \ldots, k_d \geq 2$, Condition~\eqref{condition:-1} of Theorem~\ref{thm:brwcondition} is satisfied by choosing $a = 1$ and $m = 2$.
This completes the proof.  
\end{proof}

\section{Branched covers of torus knots}\label{sec:torus}

Since the cyclic branched cover of a torus knot is Seifert fibered, to determine if it is an excellent manifold our general strategy will be to check the existence of horizontal foliations by Theorem~\ref{thm:brwcondition} and appeal to Theorem~\ref{thm:sfequiv}.  

Thus, in order to do this, we must compute the Seifert invariants of these manifolds.  In fact, these have been explicitly calculated by N\'u\~{n}ez and Ram\'irez-Losada \cite[Theorem 1]{NR}.  Due to its length, we do not state it here.  However, we will use their result throughout, so we refer the reader to \cite{NR} for the precise statements.  We will switch between the Seifert invariants of $\Sigma_n(T_{p,q})$ and those of $-\Sigma_n(T_{p,q})$ by \eqref{eqn:seifertorientation} to make use of both Conditions~\eqref{condition:-1} and \eqref{condition:-(n-1)} of Theorem~\ref{thm:brwcondition}.  

We begin with a quick lemma that will help simplify our case analysis.  

\begin{lemma}\label{lem:excellentbranchedseifert}
Suppose that $\Sigma_r(T_{p,q})$ is an excellent manifold and let $r$ divide $n$.  Then $\Sigma_n(T_{p,q})$ is an excellent manifold.  
\end{lemma}
\begin{proof}
By assumption, $\pi_1(\Sigma_r(T_{p,q}))$ is left-orderable.  As discussed in Section~\ref{subsec:branchedcovers}, there exists a non-zero degree map from $\Sigma_n(T_{p,q})$ to $\Sigma_r(T_{p,q})$.  Because $\Sigma_n(T_{p,q})$ is prime and $\pi_1(\Sigma_r(T_{p,q}))$ is left-orderable, Theorem~\ref{thm:brw} shows that $\pi_1(\Sigma_n(T_{p,q}))$ is left-orderable.  The result now follows from Theorem~\ref{thm:sfequiv}.  
\end{proof}

We now dispense with the easiest case of Theorem~\ref{thm:torus}; we state this case separately since we will regularly appeal to it in the proof for the general case.  
\begin{proposition}\label{prop:relativelyprime}
If $\gcd(n,pq) = 1$, then $\Sigma_n(T_{p,q})$ is excellent, unless $\{p,q,n\} = \{2,3,5\}$. 
\end{proposition}
\begin{proof}
If $\gcd(n,pq) = 1$, then $\Sigma_n(T_{p,q})$ is a Seifert fibered integer homology sphere with base orbifold $S^2$ and three singular fibers of multiplicities $p$, $q$, and $n$.  It follows from \cite[Corollary 3.12]{BRW} that $\pi_1(\Sigma_n(T_{p,q}))$ is left-orderable unless $\{p,q,n\} = \{2,3,5\}$.  By Theorem~\ref{thm:sfequiv}, $\Sigma_n(T_{p,q})$ is excellent unless $\{p,q,n\} = \{2,3,5\}$.  
\end{proof}

With a view towards applying Lemma~\ref{lem:excellentbranchedseifert}, we establish the following.  
\begin{proposition}\label{prop:primefactor}
Suppose that $r$ divides either $p$ or $q$.  Then $\Sigma_r(T_{p,q})$ is excellent unless $r = 2$ and $\{p,q\} = \{2,k\}$ or $r = 2$ and $\{p,q\} = \{3,4\}$ or $r = 3$ and $\{p,q\} = \{2,3\}$.
\end{proposition}
\begin{proof}
Let $r$ divide $p$ or $q$.  Without loss of generality, we may assume that $r$ divides $p$.  Let $\beta_1$, $\beta_2$ be such that 
\begin{equation}\label{eqn:pq}
\beta_1 q + \beta_2 p = -1.  
\end{equation}
By case (2) of \cite[Theorem 1]{NR}, we have 
\begin{equation}\label{eqn:seifertbranched}
\Sigma_r(T_{p,q}) = M\left( \frac{\beta_1}{p/r}, \frac{\beta_2}{q},\ldots,\frac{\beta_2}{q}\right), 
\end{equation}
where there are $r$ fibers of the form $\frac{\beta_2}{q}$.  

We may choose $\beta_2$ such that $0 < \beta_2 < q$.  Then, by \eqref{eqn:pq}, we have 
\begin{equation}\label{eqn:beta1}
-(p-1) \leq \beta_1 \leq -1.
\end{equation}

Therefore, $\frac{r\beta_1}{p}$ satisfies 
\[
\frac{-(p-1)r}{p} \leq \frac{r\beta_1}{p} \leq \frac{-r}{p}.
\]
It follows that after normalizing the Seifert invariants, we have 
\begin{equation}\label{eqn:rb}
-r \leq b \leq -1.
\end{equation}

\noindent
{\bf Case 1:} $r < p$ \\
In this case, the number of exceptional fibers of $\Sigma_r(T_{p,q})$ is $(r+1)$.  By \eqref{eqn:rb}, unless $b = -1$ or $b = -r$, we have satisfied Condition~\eqref{condition:fibers} of Theorem~\ref{thm:brwcondition}, and therefore $\Sigma_r(T_{p,q})$ admits a horizontal foliation.  By Theorem~\ref{thm:sfequiv}, this manifold is excellent.  \\

{\bf Subcase (i):} $b = -1$ \\
Since $0 < \beta_2 < q$ and $b = -1$, we must have $-1 < \frac{r\beta_1}{p} < 0$, giving $\beta_1 > -\frac{p}{r}$ and $\frac{r\beta_1}{p} = -1 + \frac{p + r\beta_1}{p}$, where $0 < p + r\beta_1 < p$.  Thus, by \eqref{eqn:seifertbranched}, we have $\Sigma_r(T_{p,q}) = M(-1,\frac{\beta_1 + p/r}{p/r}, \frac{\beta_2}{q}, \ldots,\frac{\beta_2}{q})$.  Let $N$ be the least positive integer such that 
\[
\beta_1 < -\frac{p}{r^N}.
\]
Clearly we have $N \geq 2$.  
We first claim that $\frac{p + r\beta_1}{p} < \frac{r^{N-1} - 1}{r^{N-1}}$.  If not, then we would have 
\[
r^{N-1}(p + r\beta_1) \geq p(r^{N-1} - 1), 
\]
and thus $r^N \beta_1 \geq -p$.  This gives $\beta_1 \geq -\frac{p}{r^N}$, which is a contradiction.  

We next claim that $\frac{\beta_2}{q} < \frac{1}{r^{N-1}}$.  By \eqref{eqn:pq}, we see that $\frac{\beta_2}{q} < -\frac{\beta_1}{p}$.  If the claim were false, we would have $-\frac{\beta_1}{p} > \frac{1}{r^{N-1}}$, and so $\beta_1 < -\frac{p}{r^{N-1}}$.  This contradicts the choice of $N$.  

These two claims combine to show that Condition~\eqref{condition:-1} of Theorem~\ref{thm:brwcondition} is satisfied by taking $m = r^{N-1}$ and $a = 1$.  This shows that $\Sigma_r(T_{p,q})$ is excellent and completes the proof for Subcase (i).  \\

{\bf Subcase (ii):}  $b = -r$ \\
In this case, we will show that Condition~\eqref{condition:-(n-1)} of Theorem~\ref{thm:brwcondition} is satisfied for $\Sigma_r(T_{p,q})$, unless $r = 2$, $p = 4$, and $q = 3$.  In other words, we consider Condition~\eqref{condition:-1} of Theorem~\ref{thm:brwcondition} for $-\Sigma_r(T_{p,q})$.  To obtain the normalized Seifert invariants for $-\Sigma_r(T_{p,q}) = \Sigma_r(T_{-p,q})$, we use the condition
\begin{equation}\label{eqn:qp}
p' q + q' p = 1.
\end{equation}
Again, we choose $q'$ so that $0 < q' < q$.  We have that $-\Sigma_r(T_{p,q})$ has (not yet normalized) Seifert invariants $M(\frac{rp'}{p},\frac{q'}{q},\ldots,\frac{q'}{q})$.    

Since $\Sigma_r(T_{p,q})$ has $b = -r$ and $r+1$ singular fibers, we have that for $-\Sigma_r(T_{p,q})$, $b = -1$.   The condition $b = -1$ is equivalent to $-1 < \frac{rp'}{p} < 0$, and thus it follows that the normalized Seifert invariants for $-\Sigma_r(T_{p,q})$ are $M(-1,\frac{p' + p/r}{p/r},\frac{q'}{q},\ldots,\frac{q'}{q})$.  We can satisfy Condition~\eqref{condition:-1} of Theorem~\ref{thm:brwcondition} if we can show there exists $m \geq 2$ such that 
\[
\frac{q'}{q} < \frac{1}{m}, \; \; \;  \frac{p+rp'}{p} < \frac{m-1}{m},
\]
or equivalently,  
\begin{equation}\label{eqn:beta2q}
\frac{q'}{q} < \frac{1}{m}, \; \; \;  \frac{rp'}{p} < - \frac{1}{m}.
\end{equation}

$ \quad \; \;$ {\bf Subcase (a):} $p' \neq -1$ \\
In this case, $p' \leq -2$.  Let $m$ be the least positive integer such that the second inequality in \eqref{eqn:beta2q} holds, i.e.  $\frac{rp'}{p} < -\frac{1}{m}$.  Then $m \geq 2$ and 
\begin{equation}\label{eqn:rbeta1}
\frac{rp'}{p} \geq -\frac{1}{m-1}.
\end{equation}
Also, since $p' < -1$, we have $\frac{rp'}{p} < -\frac{1}{p/r}$, and therefore 
\begin{equation}\label{eqn:mpr}
m \leq \frac{p}{r},
\end{equation}
by the minimality of $m$.  Recall that we are interested in showing that the first inequality in \eqref{eqn:beta2q} holds for our choice of $m$.  Because 
\[
\frac{q'}{q} = \frac{1}{pq} - \frac{p'}{p} \leq \frac{1}{pq} + \frac{1}{r(m-1)}, 
\]
by \eqref{eqn:rbeta1}, it suffices to show that $\frac{1}{pq} + \frac{1}{r(m-1)} < \frac{1}{m}$, or equivalently that 
\[
\frac{m}{pq} < 1 - \frac{m}{r(m-1)}. 
\]
Now, by \eqref{eqn:mpr}, 
\[
\frac{m}{pq} \leq \frac{p/r}{pq} = \frac{1}{rq}.  
\]
Therefore, it is enough to show that $\frac{1}{rq} < 1 - \frac{m}{r(m-1)}$, i.e., that $\frac{1}{q} < r - \frac{m}{m-1} = (r-1) + \frac{1}{m-1}$.  Since $r \geq 2$, this is clearly true.  \\

$ \quad \; \;$ {\bf Subcase (b):} $p' = -1$ \\
In this case, $-q + q' p = 1$, giving 
\begin{equation}\label{eqn:bpq}
q' p = q + 1.
\end{equation}
Let $m = \frac{p}{r} + 1$.  Then, 
\[
\frac{rp'}{p} = \frac{-r}{p} = \frac{-1}{p/r} < \frac{-1}{m},   
\]
so the first inequality in \eqref{eqn:beta2q} holds. \\ \\
\noindent {\bf Claim:} $\frac{q'}{q} < \frac{1}{m}$ unless $r = 2$, $p = 4$, and $q = 3$.  \\ 
\begin{proof}[Proof of Claim]  Since $\frac{q'}{q} = \frac{q+1}{qp}$, we must show that 
\[
\frac{q+1}{qp} < \frac{1}{m} = \frac{1}{p/r + 1} = \frac{1}{p}\left( \frac{1}{\frac{1}{r} + \frac{1}{p}}\right), 
\]
or equivalently that $\frac{q}{q+1} > \frac{1}{r} + \frac{1}{p}$.  Note that if $q = 2$, \eqref{eqn:bpq} gives that $p = 3$, contradicting the fact that $r < p$.  Therefore $q \geq 3$ and we have $\frac{q}{q+1} \geq \frac{3}{4}$.  On the other hand, 
\[
\frac{1}{r} + \frac{1}{p} < \frac{1}{2} + \frac{1}{4} = \frac{3}{4},
\]
unless $r = 2$ and $p = 4$.  Finally, if $r = 2$ and $p = 4$, then $\frac{1}{r} + \frac{1}{p} = \frac{3}{4}$ and $\frac{q}{q+1} > \frac{3}{4}$ if $q \geq 5$.  This completes the proof of the claim.  
\end{proof}
Therefore, we have shown that Condition~\eqref{condition:-1} of Theorem~\ref{thm:brwcondition} is satisfied by choosing $m = \frac{p}{r} + 1$, excluding the case $r = 2$, $p = 4$, $q = 3$.  This completes the proof of the theorem in Case 1.  \\

\noindent {\bf Case 2:} $r = p$ \\
We will show that Condition \eqref{condition:-1} of Theorem~\ref{thm:brwcondition} is satisfied for either $\Sigma_p(T_{p,q})$ or $-\Sigma_p(T_{p,q})$, except in the cases $p = 2, r = 2$ or $p = 3, q = 2, r =3$.  

We return to the setup before Case 1.  From \eqref{eqn:seifertbranched}, $r = p$ implies the normalized Seifert invariants of $\Sigma_p(T_{p,q})$ are 
\begin{equation}\label{eqn:toruscase2}
\Sigma_p(T_{p,q}) = M\left(\beta_1,\frac{\beta_2}{q},\ldots,\frac{\beta_2}{q}\right),
\end{equation}
i.e. $\beta_1 =b$.  Thus, the number of exceptional fibers of $\Sigma_p(T_{p,q})$ is $p$.  If $p = 2$, then $\Sigma_p(T_{p,q})$ is a lens space, and thus $\pi_1(\Sigma_p(T_{p,q}))$ is finite.  Therefore, suppose $p \geq 3$.  By \eqref{eqn:beta1}, we have $-(p-1) \leq b \leq -1$.  Again, Condition~\eqref{condition:fibers} of Theorem~\ref{thm:brwcondition} is satisfied unless $b = -1$ or $b = -(p-1)$.  \\

{\bf Subcase (i):} $b = -1$ \\
From \eqref{eqn:toruscase2}, we have that $\Sigma_p(T_{p,q})$ has normalized Seifert invariants $M(-1,\frac{\beta_2}{q},\ldots,\frac{\beta_2}{q})$.  In this case, since $\beta_1 = -1$, \eqref{eqn:pq} implies that $\beta_2 = \frac{q-1}{p}$.  Therefore, $\frac{\beta_2}{q} = \frac{q-1}{pq} < \frac{1}{p}$.  Condition~\eqref{condition:-1} of Theorem~\ref{thm:brwcondition} is now satisfied by taking $m = p$ and $a = 1$.  \\

{\bf Subcase (ii):} $b = -(p-1)$ \\
In this case, \eqref{eqn:pq} implies that $\beta_2p = (p-1)q-1$, and hence $\frac{q-\beta_2}{q} = \left(\frac{q+1}{q} \right) \left( \frac{1}{p} \right)$.  Thus, after reversing orientation, we see from \eqref{eqn:seifertorientation} and \eqref{eqn:toruscase2} that $-\Sigma_p(T_{p,q})$ has normalized Seifert invariants 
\[
-\Sigma_p(T_{p,q}) = M\left(-1, \left(\frac{q+1}{q} \right) \frac{1}{p} ,\ldots,\left(\frac{q+1}{q} \right) \frac{1}{p} \right ).  
\]
Unless $p = 3$ and $q = 2$ (and consequently $r = 3$), Condition~\eqref{condition:-1} of Theorem~\ref{thm:brwcondition} is satisfied by taking $m = 2, a = 1$.  

We have now shown that $\Sigma_p(T_{p,q})$ is excellent when $r$ divides $p$ or $q$, except for the cases stated in the proposition.  
\end{proof}

With Proposition~\ref{prop:relativelyprime} and Proposition~\ref{prop:primefactor}, we are now ready to prove Theorem~\ref{thm:torus}.  

\begin{proof}[Proof of Theorem~\ref{thm:torus}]
We are interested in the general case of when $\Sigma_n(T_{p,q})$ is excellent for arbitrary $n$.  We first note that if $\pi_1(\Sigma_n(T_{p,q}))$ is finite, then this group is not left-orderable.  In this case, $\Sigma_n(T_{p,q})$ is a total L-space by Theorem~\ref{thm:sfequiv}. (This was originally established in \cite[Proposition 2.3]{OSLens}).  

By Proposition~\ref{prop:relativelyprime}, we have reduced the problem to the case $\gcd(n,pq) \neq 1$.  Without loss of generality, we can assume that $\gcd(n,p) \neq 1$.  By Proposition~\ref{prop:primefactor}, we established that if $n$ has a factor $r$ which divides $p$, then $\Sigma_r(T_{p,q})$ is excellent unless $p = 4, q = 3$ or $p = 3, q = 2$, or $p = 2$.  By Lemma~\ref{lem:excellentbranchedseifert}, we thus have that $\Sigma_n(T_{p,q})$ is excellent unless $p = 4, q = 3$ or $p = 3, q = 2$ or $p = 2$.  From now on, we assume that $(p,q)$ is of one of these three forms.  Write $n = 2^k 3^\ell 5^m w$, where $\gcd(w,30) = 1$.  First, suppose $w \neq 1$.  If $\gcd(q,w) = 1$, then $\Sigma_w(T_{p,q})$ is excellent from Proposition~\ref{prop:relativelyprime}, and consequently so is $\Sigma_n(T_{p,q})$ by Lemma~\ref{lem:excellentbranchedseifert}.  Now suppose $\gcd(q,w) \neq 1$ and let $s$ be a prime factor of $\gcd(q,w)$; note that in this case $s \geq 7$ and thus $q \geq 7$.  Thus, $\Sigma_s(T_{p,q})$ is excellent by Proposition~\ref{prop:primefactor}.  By Lemma~\ref{lem:excellentbranchedseifert}, $\Sigma_n(T_{p,q})$ is excellent.  Thus, we now assume $n = 2^k 3^\ell 5^m$.  We will compile a list of the remaining cases we must check by hand.  

Let $p = 4, q = 3$.  Note that $\Sigma_5(T_{3,4})$ is excellent by Proposition~\ref{prop:relativelyprime}.  Therefore, if $m \geq 1$, $\Sigma_n(T_{4,3})$ is excellent by Lemma~\ref{lem:excellentbranchedseifert}.  Thus, assume $m = 0$.  Note that since $\gcd(n,p) \neq 1$ by assumption, we must have that $k \geq 1$.  Further, we have that $\Sigma_3(T_{4,3})$ is excellent by Proposition~\ref{prop:primefactor}.  Consequently if $\ell \geq 1$, Lemma~\ref{lem:excellentbranchedseifert} shows that $\Sigma_n(T_{4,3})$ is excellent.  Thus, for the case of $p = 4, q = 3$, it remains to consider $n = 2^k$ for $k \geq 2$.  Thus, by Lemma~\ref{lem:excellentbranchedseifert}, it suffices to show that $\Sigma_4(T_{4,3})$ is excellent, since $\pi_1(\Sigma_2(T_{4,3}))$ is finite.  This case is handled by Proposition~\ref{prop:primefactor}.  

Now, we consider the case $p = 3, q = 2, n = 2^k 3^\ell 5^m$.  Since by assumption $\gcd(n,p) \neq 1$, we have $\ell \geq 1$.  Proposition~\ref{prop:relativelyprime} shows that $\Sigma_{25}(T_{3,2})$ is excellent, and thus $\Sigma_n(T_{3,2})$ is excellent if $m \geq 2$.  Observe that $\Sigma_n(T_{3,2})$ has finite fundamental group if $n = 3$.  On the other hand, if we can show $\Sigma_n(T_{3,2})$ is excellent for $n = 6$, $9$, and $15$, this will complete the proof for the case of $p = 3, q = 2$ by again applying Lemma~\ref{lem:excellentbranchedseifert}.  

Finally, suppose that $p = 2, n = 2^k 3^\ell 5^m$.  By assumption, $k \geq 1$.  The relevant triples with finite fundamental group are $(n,p,q) = (4,2,3)$ and $(2,2,q)$ for $q \geq 3$.  First, consider the case $q = 3$.  Arguments similar to those above show that it suffices to establish the excellence of $\Sigma_n(T_{2,3})$ for $n = 8$ and $10$.  Now suppose that $q \geq 5$.  If $\gcd(q,5) = 1$, then $\Sigma_5(T_{2,q})$ is excellent by Proposition~\ref{prop:relativelyprime}.  If $5$ divides $q$, then $\Sigma_5(T_{2,q})$ is excellent by applying Proposition~\ref{prop:primefactor}.  Thus, by Lemma~\ref{lem:excellentbranchedseifert}, if $m \geq 1$, then $\Sigma_n(T_{2,q})$ is excellent.  Thus, assume $n = 2^k 3^\ell$ with $k \geq 1$.  If $q \geq 7$, $\Sigma_3(T_{2,q})$ is excellent by applying either Proposition~\ref{prop:relativelyprime} or Proposition~\ref{prop:primefactor} with the roles of $p$ and $q$ reversed (dependent on whether $\gcd(3,q) = 1$).  Therefore, if $q \geq 7$ and $\ell \geq 1$, $\Sigma_n(T_{2,q})$ is excellent by Lemma~\ref{lem:excellentbranchedseifert}.  Thus, for $p = 2, q \geq 7$, it suffices to show that $\Sigma_4(T_{2,q})$ is excellent.  On the other hand, if $q = 5$, then it suffices to show $\Sigma_6(T_{2,5})$ and $\Sigma_4(T_{2,5})$ are excellent.  

We summarize the above discussion with the list of remaining cases for which we need to establish excellence to complete the proof of the theorem:  

\begin{enumerate}
\item $p= 2, q \geq 5, n = 4$ 
\item $p = 2, q = 3, n = 6$
\item $p = 2, q = 3, n = 8$
\item $p = 2, q = 3, n = 9$
\item $p = 2, q = 3, n = 10$
\item $p = 2, q = 3, n = 15$
\item $p = 2, q = 5, n = 6$.
\end{enumerate}

\noindent \\ {\bf Case 1:} $p= 2, q \geq 5, n = 4$ \\
To compute the Seifert invariants of $\Sigma_4(T_{2,q})$, we use case (3)(a) of \cite[Theorem 1]{NR}.  Write $q = 2k-1, k > 2$.  Let $q^* = -1$; then, $qq^* = 1-2k$, and so $k$ is as in \cite[Theorem 1]{NR}.  Also, we have $\beta_1 q + \beta_2 p = -1$, and so we can take $\beta_1 = 1, \beta_2 = -k$.  From \cite[Theorem 1]{NR}, we get that 
\[
\Sigma_4(T_{2,q}) = M\left(\frac{1}{2}, k, -\frac{k^2}{q}, -\frac{k^2}{q}\right).
\]  
Let $c = \lfloor \frac{k^2}{q} \rfloor + 1$.  Then, $b = k-2c$ and $\Sigma_4(T_{2,q})$ has normalized Seifert invariants
\begin{equation}\label{eqn:sigma42q}
\Sigma_4(T_{2,q}) = M\left(k-2c, \frac{1}{2}, c - \frac{k^2}{q}, c - \frac{k^2}{q}\right).  
\end{equation}
Note that
\[
c = \left \lfloor \frac{k^2}{2k-1} \right \rfloor + 1 = \left\{ \begin{array}{rl} \frac{k}{2} + 1, & k \text{ even} \\ \frac{k+1}{2}, & k \text{ odd}. \end{array}\right.
\]

$ \quad \; \;$ {\bf Subcase (i):} $k$ odd \\
Then $b = k - 2c = -1$.  Also, in this case
\begin{align*}
c - \frac{k^2}{q} = \frac{k+1}{2} - \frac{k^2}{q} & = \frac{q(k+1) - 2k^2}{2q} \\
& = \frac{k-1}{4k-2} \\
& < \frac{1}{4}.  
\end{align*}
Therefore, Condition~\eqref{condition:-1} of Theorem~\ref{thm:brwcondition} is satisfied by taking $m = 4, a = 1$.  \\

$ \quad \; \;$ {\bf Subcase (ii):} $k$ even \\
Here $b = k - (k+2) = -2$.  Also, in this case
\begin{align*}
1 - (c - \frac{k^2}{q}) = (1 - c) + \frac{k^2}{q} & = \frac{k^2}{q} - \frac{k}{2} \\
& = \frac{2k^2 - qk}{2q} \\
& = \frac{k}{4k-2}.   
\end{align*}
It follows from \eqref{eqn:seifertorientation} and \eqref{eqn:sigma42q} that $-\Sigma_4(T_{2,q})$ has normalized Seifert invariants $M(-1,\frac{1}{2}, \frac{k}{4k-2},\frac{k}{4k-2})$.  If $q \geq 5$, then $k > 2$ and so $\frac{k}{4k-2} < \frac{1}{3}$.  Condition~\eqref{condition:-1} of Theorem~\ref{thm:brwcondition} is therefore satisfied by taking $m = 3, a = 1$.  \\

\noindent {\bf Case 2:} $p = 2, q = 3, n = 6$ \\
By case (2) of \cite[Theorem 1]{NR}, we have that $H_1(\Sigma_6(T_{2,3}))$ is infinite.  Therefore, $\Sigma_6(T_{2,3})$ is excellent. \\

\noindent {\bf Case 3:} $p= 2, q = 3, n = 8$ \\
By case (3) of \cite[Theorem 1]{NR}, letting $\beta_1 = -1, \beta_2 = 1, k = 1$, and $3^* = -1$ as in \cite[Theorem 1]{NR}, we have 
\[
\Sigma_8(T_{2,3}) = M\left(-1,\frac{1}{4}, \frac{1}{3},\frac{1}{3}\right).
\]
Condition~\eqref{condition:-1} of Theorem~\ref{thm:brwcondition} holds with $m = 2, a = 1$. \\
 
\noindent {\bf Case 4:} $p = 2, q = 3, n = 9$ \\
By case (3)(a) of \cite[Theorem 1]{NR}, letting $\beta_1 = 1, \beta_2 = -1, k = 1,$ and $2^* = -1$ as in \cite[Theorem 1]{NR}, we have  
\[
\Sigma_9(T_{3,2}) = M\left(\frac{1}{3}, 1, -\frac{1}{2}, -\frac{1}{2}, -\frac{1}{2}\right) = M\left(-2,\frac{1}{3},\frac{1}{2},\frac{1}{2},\frac{1}{2}\right).  
\]
Therefore, $b = -2$ and there are four singular fibers.  Condition~\eqref{condition:fibers} of Theorem~\ref{thm:brwcondition} now holds.  \\

The remaining cases can be handled the same as in Case 3 or Case 4.  We leave these to the reader. 
\end{proof}

\section{Cables}\label{sec:cables}
Let $K$ be a knot in $S^3$, with regular neighborhood $N(K)$.  Let $C$ be a simple loop on $\partial N(C)$ with slope $p/q$, where $q\ge 2$.  Then $C$ is the {\em $(p,q)$-cable} of $K$, $C_{p,q}(K)$.  If $K$ is trivial then $C_{p,q} (K)$ is the torus knot $T_{p,q}$.  Since $C_{-p,q} (K) = C_{p,-q}(K) = - C_{p,q}(-K)$, we will always assume that $p \geq1$ and $q\ge 2$.

We recall the main theorem about cables that we are interested in proving.  \\ \\
{\bf Theorem~\ref{thm:cables}.}  {\em Let $K$ be a non-trivial knot in $S^3$.  Then, $\Sigma_n(C_{p,q}(K))$ is excellent, unless $n = q = 2$.} \\

In the case $n=q=2$, it turns out that $\Sigma_2 (C_{p,2} (K)) \cong X_1 \cup_\partial X_2$, where $X_i$ is a copy of the exterior $X$ of $K$, $i=1,2$.  For $p=1$ we get the following conditional result.

\begin{theorem}\label{thm:conditional}
If $\frac{1}{2}$ is a CTF slope for $K$ then $\Sigma_2 (C_{1,2} (K))$ is excellent.
\end{theorem}
In particular, if the Li-Roberts Conjecture holds, then $\Sigma_2(C_{1,2}(K))$ is excellent for all non-trivial $K$.  The theorem will be proved at the end of the section.  

Let $S^3 = V\cup_T W$ be the standard genus~1 Heegaard splitting of $S^3$.  Let $C\subset \text{int }W$ be an isotopic copy of a $(p,q)$-curve on $T$, where this curve intersects the meridional disk of $W$ $q$ times.  Let $J$ be a core of $V$.  Let $\sigma,\tau \subset T$ be meridians of $V,W$, respectively.  We will be more precise about our orientations of these curves shortly.

In $S^3 = V\cup W$ the curve $C$ is the torus knot $T_{p,q}$.  In the corresponding Seifert fibration of $S^3$, in which $C$ is an ordinary fiber, $J$ is 
the exceptional fiber of multiplicity $p$.  Let $\pi : \Sigma_n (T_{p,q}) \to S^3$ be the branched covering projection.  Then $\pi^{-1}(V)$ has $\gcd(n,q)$ components $\widetilde V_i$, $1\le i\le \gcd(n,q)$, each an $\frac{n}{\gcd(n,q)}$-fold covering of $V$.  In the induced Seifert fibration on $\Sigma_n (T_{p,q})$ the core of each $\widetilde V_i$ is an exceptional fiber of multiplicity $\frac{p}{\gcd(n,p)}$. Let $\widetilde \sigma_i$ be a meridian of $\widetilde V_i$, $1\le i\le \gcd(n,q)$. Also, $\widetilde\tau_i = \pi^{-1} (\tau) \cap \partial\widetilde V_i$ is connected for $1\le i\le \gcd(n,q)$.  Let $s, \varphi \subset \partial W$ be an oriented section-fiber pair for the induced Seifert structure on $W$ (i.e. $s \cdot \varphi = 1$).  We orient $\tau$ such that $\tau \cdot \varphi = q$ in $\partial W$.  Observe that $\varphi$ is a $(p,q)$-curve on $\partial W$.  We then orient $\sigma$ such that $\tau \cdot \sigma = 1$ on $\partial W$.

Let $X$ be the exterior of the knot $K$, and let $\lambda,\mu$ be a ($0$-framed) longitude-meridian pair on $\partial X$.  If we remove $int(V)$ from $S^3$ and replace it by $X$, identifying $\partial X$ with $\partial W$ in such a way that $\lambda,\mu$ are identified with $\sigma,\tau$, respectively, we get $S^3$, and the curve $C$ becomes the $(p,q)$-cable of $K$.

Let $\widetilde W = \pi^{-1}(W) = \Sigma_n (T_{p,q}) - \coprod_{i=1}^{\gcd(n,q)} int(\widetilde V_i)$, the $n$-fold cyclic branched cover of $W$ branched along $C$. 
Let $\widetilde X$ be the $\frac{n}{\gcd(n,q)}$-fold cyclic covering of $X$, and let $\widetilde\lambda, \widetilde\mu \subset \partial \widetilde X$ be as in the discussion in Section~\ref{subsec:branchedcovers}.  Then $\Sigma_n (C_{p,q}(K)) = \widetilde W \cup \coprod_{i=1}^{\gcd(n,q)} \widetilde X_i$, where $\widetilde X_i$ is a copy of $\widetilde X$ and $\partial\widetilde X_i$ is glued to $\partial\widetilde V_i$ in such a way that $\widetilde\lambda_i,\widetilde \mu_i$ are identified with $\widetilde\sigma_i,\widetilde\tau_i$, respectively.

\subsection{The proof of Theorem~\ref{thm:cables}}
The following proposition establishes Theorem~\ref{thm:cables} in the generic case.

\begin{proposition}\label{prop:generic}
Let $K$ be a knot in $S^3$.  Suppose $p\ne 1$ and $\{ n,p,q\} \ne \{2,2,r\}$, $\{ 2,3,3\}$, $\{2,3,4\}$ or $\{2,3,5\}$.  Then $\Sigma_n (C_{p,q}(K))$ is excellent.
\end{proposition}
\begin{proof}
Using the notation in the discussion before the statement of the proposition, $\widetilde W(\widetilde{\boldsymbol\sigma}) = \widetilde W (\widetilde\sigma_1,\ldots, \widetilde\sigma_{\gcd(n,q)}) = \Sigma_n (T_{p,q})$, which by Theorem~\ref{thm:torus} is excellent if  $\{n,p,q\}$ is not one of the exceptions listed in the proposition. Hence by Lemma~\ref{lem:seifertboundary}, $\widetilde{\boldsymbol\sigma}$ is excellent for $\widetilde W$. By \cite{Gabai3}, $\lambda$ is a CTF slope for $X$, and hence $\widetilde\lambda$ is a CTF  slope for $\widetilde X$.  In particular, we have that $\widetilde X(\widetilde\lambda)$ is prime. Since $H_1(X(\lambda))$ is infinite, so is $H_1(\widetilde X(\widetilde\lambda))$.  Therefore, $\widetilde \lambda$ is an LO slope for  $\widetilde X$ by \cite[Corollary 3.4]{BRW}.  Thus $\Sigma_n (C_{p,q}(K))$ is excellent by Lemma~\ref{lem:excellentgluing}.
\end{proof}

To treat the cases where $(n,p,q)$ is one of the exceptional triples in Theorem~\ref{thm:torus} we continue with a more detailed analysis of the description of $\Sigma_n(C_{p,q}(K))$ given just before the present subsection.  

Recall that $\varphi \subset \partial W$ is a regular fiber in the Seifert fibration of $S^3$ described above.  Note that $\widetilde\varphi_i = \pi^{-1}(\varphi) \cap \partial \widetilde{V}_i$ has $\gcd(\frac{n}{\gcd(n,q)},p) = \gcd(n,p)$ components.  Since we have oriented $\varphi$ so that $\tau \cdot \varphi = q$ on $\partial W$, we have $\widetilde{\tau}_i \cdot \widetilde{\varphi_i} = \frac{nq}{\gcd(n,q)\gcd(n,p)}$ on $\partial \widetilde{W}$.  For notation, we let $\omega = \frac{nq}{\gcd(n,q)\gcd(n,p)}$.    

We have $\widetilde{W}(\widetilde{\sigma}_1,\ldots,\widetilde{\sigma}_{\gcd(n,q)}) = \Sigma_n(T_{p,q})$.  As in Section~\ref{sec:torus}, we will use the description of the Seifert invariants of cyclic branched covers of torus knots in \cite[Theorem 1]{NR}.  We define $\widetilde{\gamma}_i \subset \partial \widetilde{V}_i$ as the section with respect to which the Seifert invariants of $\Sigma_n(T_{p,q})$ are described in \cite[Theorem 1]{NR}.  Thus $\widetilde{\gamma}_i \cdot \widetilde{\varphi}_i = 1$ in $\partial \widetilde{W}$, and if the core of $\widetilde{V}_i$ has Seifert invariants $\frac{r}{s}$ in $\Sigma_n(T_{p,q})$, then $\widetilde{\sigma}_i = \pm (s\widetilde{\gamma}_i + r\widetilde{\varphi}_i)$.  In some cases, the orientation of $\widetilde{\sigma}_i$ will be  easily determined, while often it will not.  In the latter situation, rather than repeat the constructions of \cite{NR} in precise detail to obtain the exact orientation of $\widetilde{\sigma}_i$, it will be easier to simply treat both cases, even though only one can possibly arise.      

Observe that $\widetilde{\tau}_i = \omega \widetilde{\gamma}_i + \eta \widetilde{\varphi}_i$ for some $\eta \in \mathbb{Z}$.  Recall that we must have that $\widetilde{\tau}_i \cdot \widetilde{\sigma}_i = \widetilde{\gamma}_i \cdot \widetilde{\varphi}_i = 1$ on $\partial \widetilde{W}$.  This will determine the possible values of $\eta$.  

The strategy to complete the proof of Theorem~\ref{thm:cables} is as follows.  We consider slopes $\alpha_k = \mu + k\lambda$ on $\partial X$, $k \in \Z$.  The corresponding slope on $\partial \widetilde{X}_i$ is $\widetilde{\alpha}_k = \widetilde{\mu}_i + \frac{n}{\gcd(n,q)} k \widetilde{\lambda}_i$, which is identified with $\ttau_i + \frac{n}{\gcd(n,q)} k \widetilde{\sigma}_i$ on $\partial \widetilde{W}$.  We will show that for $k \gg 0$ or $k \ll 0$,  $\widetilde{W}(\widetilde{\alpha}_k,\ldots,\widetilde{\alpha}_k)$ is a Seifert fibered space with a horizontal foliation, and therefore $(\widetilde{\alpha}_k,\ldots,\widetilde{\alpha}_k)$ is an excellent multislope for $\widetilde{W}$ by Lemma~\ref{lem:seifertboundary}.  Since, for $|k|$ sufficiently large, $\widetilde{\alpha}_k$ is an excellent slope for $\widetilde{X}$ by Lemma~\ref{lem:cover1/n}, $\Sigma_n(C_{p,q}(K))$ will be excellent by Lemma~\ref{lem:excellentgluing}.  

The details are given in Propositions~\ref{prop:exceptional}, \ref{prop:234}, and \ref{prop:1q} below.  

We first treat the case where $\gcd(n,q) = 1$ and $p \neq 1$.  Then, $\pi^{-1}(V) = \widetilde{V}$ is connected, and we drop the subscript $i$ from $\widetilde{V}_i, \widetilde{\sigma}_i$, etc.  Also, for the rest of this section, $X^{(n)}$ will denote the $n$-fold cyclic covering of $X$.  

\begin{proposition}\label{prop:exceptional}
Let $K$ be a non-trivial knot in $S^3$.  If $\{n,p,q\} = \{2,3,5\}$ or $(n,p,q) = (2,4,3), (3,3,2)$, $(4,2,3)$, or $(2,2,q)$, then $\Sigma_n(C_{p,q}(K))$ is excellent.  
\end{proposition}

\begin{proof}
Note that $\widetilde{J} = \pi^{-1}(J)$ is a fiber of multiplicity $\frac{p}{\gcd(n,p)}$ in $\Sigma_n(T_{p,q})$.  Also, $\widetilde{\tau} = \omega  \widetilde{\gamma} + \eta \widetilde{\varphi}$ where $\omega = \frac{nq}{\gcd(n,p)}$, since $\gcd(n,q) = 1$.  Here, $\talpha_k = \tmu + nk\tlambda$ is identified with $\ttau + nk\tsigma$ on $\partial \widetilde{W}$.  Thus, by the discussion preceding the statement of the proposition, it suffices to show that $\widetilde{W}(\ttau + nk \tsigma)$ admits a horizontal foliation for $k \gg 0$ or $k \ll 0$.  \\

\noindent {\bf Case 1:} $(n,p,q) = (2,3,5)$.\\
From \cite[Theorem 1]{NR}, $\Sigma_2(T_{3,5}) = M(1,-\frac{1}{2},-\frac{1}{3},-\frac{1}{5})$.  Since $\widetilde{J}$ is the fiber of multiplicity 3, $\widetilde{\sigma} = \pm(-3\widetilde{\gamma} + \widetilde{\varphi})$.  We have $\omega = 10$, and thus 
\[
1 = \ttau \cdot \tsigma = \pm (10 \tgamma + \eta \tphi) \cdot (-3 \tgamma + \tphi) = \pm (3\eta + 10).
\]
Therefore, we have that $\eta = -3$, and thus $\widetilde{\tau} = 10 \widetilde{\gamma} - 3\widetilde{\varphi}$ and $\widetilde{\sigma} = -3\tgamma + \tphi$.  

The slope $\widetilde{\alpha}_k = \widetilde{\mu} + 2 k \widetilde{\lambda}$ on $\partial X^{(2)}$ is identified with $\widetilde{\tau} + 2  k \widetilde{\sigma} = (10 \widetilde{\gamma} - 3 \widetilde{\varphi}) + 2k( -3 \widetilde{\gamma} + \widetilde{\varphi}) = (10-6k) \widetilde{\gamma} + (-3+2k)\widetilde{\varphi}$ on $\partial \widetilde{W}$.  Hence $\widetilde{W}(\widetilde{\alpha}_k)$ has Seifert invariants $M(1,-\frac{1}{2},-\frac{1}{5}, \frac{2k-3}{10-6k}) = -M(-1,\frac{1}{2},\frac{1}{5}, \frac{2k-3}{6k-10})$.  If $k \leq 0$, then $0 < \frac{2k-3}{6k-10} < \frac{1}{3}$.  By choosing $m = 3$ and $a = 1$, it follows from Condition~\eqref{condition:-(n-1)} of Theorem~\ref{thm:brwcondition} that $\widetilde{W}(\widetilde{\alpha}_k)$ has a horizontal foliation.  This is what we wanted to show.  \\

\noindent {\bf Case 2:} $(n,p,q) = (2,5,3)$. \\
 Here $\widetilde{\sigma} = \pm (-5 \widetilde{\gamma} + \widetilde{\varphi})$, and $\widetilde{\tau} = 6 \widetilde{\gamma} + \eta \widetilde{\varphi}$.  We have 
\[
1 = \ttau \cdot \tsigma = \pm (6 \tgamma + \eta \tphi) \cdot (-5\tgamma + \tphi) = \pm (6 + 5 \eta).  
\]
Therefore, $\eta = -1$ and we have $\ttau = 6 \tgamma - \tphi$, $\tsigma = -5 \tgamma + \tphi$.  
Hence on $\partial \widetilde{W}$, $\widetilde{\alpha}_k = \widetilde{\tau} + 2k \tsigma = (6-10k)\tgamma + (2k-1)\tphi$.  Then $\widetilde{W}(\talpha_k)$ is the Seifert fibered space $M(1,-\frac{1}{2},-\frac{1}{3}, \frac{2k-1}{6-10k}) = - M(-1,\frac{1}{2},\frac{1}{3},\frac{2k-1}{10k-6})$.  For $k \leq 0$, $0 < \frac{2k-1}{10k-6} < \frac{1}{5}$.  Letting $m = 5$ and $a = 2$, we see that $\widetilde{W}(\talpha_k)$ has a horizontal foliation by Condition~\eqref{condition:-(n-1)} of Theorem~\ref{thm:brwcondition}. \\

\noindent {\bf Case 3:} $(n,p,q) = (3,2,5)$.\\
 Here $\tsigma = \pm (-2 \tgamma + \tphi)$, and $\ttau = 15 \tgamma + \eta \tphi$.  In this case, we have 
\[
1 = \pm (15 \tgamma + \eta \tphi) \cdot (-2 \tgamma + \tphi) = \pm (15 + 2 \eta).  
\]
Thus, we have two possibilities.  The first is that $\eta = -7$, $\ttau = 15 \tgamma - 7 \varphi$, and $\tsigma = -2\tgamma + \tphi$.  The second case is that $ \eta = -8$, $\ttau = 15\tgamma - 8 \tphi$, and $\tsigma = 2 \tgamma - \tphi$.  

First, suppose that $\eta = -7$.  The slope $\talpha_k = \tmu + 3k \tlambda$ on $\partial X^{(3)}$ is identified with $\ttau + 3k \tsigma = (15-6k)\tgamma + (3k-7)\tphi$ on $\partial \tW$.  Therefore $\tW(\talpha_k)$ has Seifert invariants  $M(1,-\frac{1}{3},-\frac{1}{5}, \frac{3k-7}{15-6k}) = -M(-1,\frac{1}{3},\frac{1}{5},\frac{3k-7}{6k-15})$.  If $k \leq 0$, then $0< \frac{7-3k}{15-6k} < \frac{1}{2}$ and $\tW(\talpha_k)$ has a horizontal foliation by Condition~\eqref{condition:-(n-1)} of Theorem~\ref{thm:brwcondition}. 

Next, suppose that $\eta = -8$.  The slope $\talpha_k = \tmu + 3k \tlambda$ on $\partial X^{(3)}$ is identified with $\ttau + 3k \tsigma = (15+6k)\tgamma + (-8-3k)\tphi$ on $\partial \tW$.  Therefore $\tW(\talpha_k)$ has Seifert invariants  $M(1,-\frac{1}{3},-\frac{1}{5}, \frac{-3k-8}{6k+15}) = -M(-1,\frac{1}{3},\frac{1}{5},\frac{3k+8}{6k+15})$.  If $k \ll 0$, then $0 < \frac{3k+8}{6k+15} < \frac{1}{2}$ and $\tW(\talpha_k)$ has a horizontal foliation by Condition~\eqref{condition:-(n-1)} of Theorem~\ref{thm:brwcondition}.

The cases $(n,p,q) = (3,5,2), (5,2,3)$, and $(5,3,2)$ are completely analogous; we leave the details to the reader.  \\

\noindent {\bf Case 4:} $(n,p,q) = (2,4,3)$.  \\
By \cite[Theorem 1]{NR}, $\Sigma_2(T_{4,3}) = M(\frac{1}{2},-\frac{1}{3},-\frac{1}{3})$.  Since $\widetilde{J}$ has multiplicity 2, we have $ \tsigma = \pm(2 \tgamma  + \tphi)$.  Also, $\omega = \frac{nq}{\gcd(n,p)} = 3$.  We again have two cases.  The first case is $\eta = 1$, $\ttau = 3 \tgamma + \tphi$, $\tsigma = 2 \tgamma + \tphi$.  The slope $\talpha_k = \tmu + 2k\tlambda$ on $\partial X^{(2)}$ is identified with $\ttau + 2k \tsigma = (4k+3)\tgamma + (2k+1)\tphi$ on $\partial \tW$.  So 
\[
\tW(\talpha_k) = M\left(-\frac{1}{3},-\frac{1}{3}, \frac{2k+1}{4k+3}\right) = -M\left(\frac{1}{3},\frac{1}{3},-\frac{2k+1}{4k+3}\right) = -M\left(-1,\frac{1}{3},\frac{1}{3},\frac{2k+2}{4k+3}\right).
\]
Since $\frac{2k+1}{4k+3} < \frac{1}{2}$ for $k \geq 0$, $\tW(\talpha_k)$ has a horizontal foliation by Condition~\eqref{condition:-(n-1)} of Theorem~\ref{thm:brwcondition}.  

The second case is $\eta = 2$, $\ttau = 3 \tgamma + 2\tphi$, $\tsigma = -2 \tgamma -\tphi$.  The slope $\talpha_k = \tmu + 2k\tlambda$ on $\partial X^{(2)}$ is identified with $\ttau + 2k \tsigma = (3 - 4k)\tgamma + (2-2k)\tphi$ on $\partial \tW$.  So 
\[
\tW(\talpha_k) = M\left(-\frac{1}{3},-\frac{1}{3}, \frac{2k-2}{4k-3}\right) = -M\left(\frac{1}{3},\frac{1}{3},\frac{2-2k}{4k-3}\right) = -M\left(-1,\frac{1}{3},\frac{1}{3},\frac{2k-1}{4k-3}\right).  
\]
Since $\frac{2k-1}{4k-3} < \frac{1}{2}$ for $k \leq 0$, $\tW(\talpha_k)$ has a horizontal foliation by Condition~\eqref{condition:-(n-1)} of Theorem~\ref{thm:brwcondition}.  \\

\noindent {\bf Case 5:} $(n,p,q) = (3,3,2)$.  \\
By \cite[Theorem 1]{NR}, $\Sigma_3(T_{3,2}) = M(-\frac{1}{2},-\frac{1}{2},-\frac{1}{2},\frac{1}{1})$.  Since $\widetilde{J}$ has multiplicity 1, we have $\tsigma = \pm (\tgamma + \tphi)$.  We have two cases.  

First, $\ttau = 2 \tgamma + \tphi$, $\tsigma = \tgamma + \tphi$.  The slope $\talpha_k = \tmu + 3k \tlambda$ on $\partial X^{(3)}$ is identified with $\ttau + 3k \tsigma = (3k+2) \tgamma + (3k+1) \tphi$ on $\partial \tW$.  Therefore, $\tW(\talpha_k)$ has Seifert invariants 
\[
M\left(-\frac{1}{2},-\frac{1}{2},-\frac{1}{2},\frac{3k+1}{3k+2}\right) = M\left(-2,\frac{1}{2},\frac{1}{2},\frac{1}{2},\frac{-1}{3k+2}\right). 
\]
For $k \ll 0$, the latter are normalized Seifert invariants, and thus $\tW(\alpha_k)$ has a horizontal foliation by Condition~\eqref{condition:fibers} of Theorem~\ref{thm:brwcondition}.   

Second, we have $\ttau = 2\tgamma + 3\tphi$, $\tsigma = -\tgamma - \tphi$.  The slope $\talpha_k = \tmu + 3k \tlambda$ on $\partial X^{(3)}$ is identified with $\ttau + 3k \tsigma = (2-3k) \tgamma + (3-3k) \tphi$ on $\partial \tW$.  Therefore, $\tW(\talpha_k)$ has Seifert invariants 
\[
M\left(-\frac{1}{2},-\frac{1}{2},-\frac{1}{2},\frac{3-3k}{2-3k}\right) = M\left(-2,\frac{1}{2},\frac{1}{2},\frac{1}{2},\frac{1}{2-3k}\right).
\]
For $k \leq 0$, the latter invariants are normalized, and thus $\tW(\alpha_k)$ has a horizontal foliation by Condition~\eqref{condition:fibers} of Theorem~\ref{thm:brwcondition}.   \\

\noindent {\bf Case 6:} $(n,p,q) = (4,2,3)$.\\
 By \cite[Theorem 1]{NR} $\Sigma_4(T_{2,3}) = M(-\frac{1}{2},\frac{1}{1},-\frac{1}{3},-\frac{1}{3})$.  Since $\widetilde{J}$ has multiplicity 1, we have $ \tsigma = \pm (\tgamma + \tphi)$.  Again, there are two cases.  
 
For the first case, we have $\ttau = 6 \tgamma + 5 \tphi$, $\tsigma = \tgamma + \tphi$.  The slope $\talpha_k = \tmu + 4k \tlambda$ on $\partial X^{(4)}$ is identified with $\ttau + 4k \tsigma = (4k+6) \tgamma + (4k+5) \tphi$ on $\partial \tW$.  Thus, $\tW(\alpha_k)$ has Seifert invariants
 \[
 M\left(-\frac{1}{2},-\frac{1}{3},-\frac{1}{3},\frac{5+4k}{6+4k}\right) = -M\left(-2,\frac{1}{2},\frac{1}{3},\frac{1}{3},\frac{4k+7}{4k+6}\right).  
\]
Hence for $k \leq -2$ we are in Condition~\eqref{condition:fibers} of Theorem~\ref{thm:brwcondition} and $\tW(\talpha_k)$ has a horizontal foliation. 

For the second case, we have $\ttau = 6 \tgamma + 7 \tphi$, $\tsigma = -\tgamma - \tphi$.  In this case, $\tW(\talpha_k)$ has Seifert invariants
\[
 M\left(-\frac{1}{2},-\frac{1}{3},-\frac{1}{3},\frac{7-4k}{6-4k}\right) = -M\left(-2,\frac{1}{2},\frac{1}{3},\frac{1}{3},\frac{4k-5}{4k-6}\right).  
\]
Hence for $k \leq 0$ we are in Condition~\eqref{condition:fibers} of Theorem~\ref{thm:brwcondition} and $\tW(\talpha_k)$ has a horizontal foliation. \\

\noindent {\bf Case 7:} $(n,p,q) = (2,2,q)$. \\
Observe that $\Sigma_2(T_{2,q})$ is a lens space.  More specifically, from \cite[Theorem 1]{NR} we get that $\Sigma_2 (T_{2,q})$ has Seifert invariants $M(-\frac11, \frac{\beta_2}q, \frac{\beta_2}q)$ where $\beta_2 =\frac{q-1}2$.  Since $J$ has multiplicity~2, and $\widetilde J\to J$ is a 2-fold connected covering, we have $\widetilde J$ is the fiber with Seifert invariant $-\frac{1}1$.  
Hence $\widetilde\sigma = \pm(-\widetilde\gamma + \widetilde\varphi)$.  Again, there are two cases.  

First, we have $\ttau = q \tgamma + (1-q) \tphi$ and $\tsigma = -\tgamma + \tphi$.  In $\Sigma_2 (C_{2,q}(K)) \cong X^{(2)} \cup \widetilde W$, the slope of $\widetilde \alpha_k = \widetilde\mu + 2k\widetilde\lambda$ is identified with $\widetilde\tau + 2k\widetilde\sigma = (q-2k)\widetilde\gamma + (2k+1-q)\widetilde\varphi$.  Then, $\tW(\alpha_k)$ has Seifert invariants 
\[
M\left(-1,\frac{(q-1)/2}{q},\frac{(q-1)/2}{q},\frac{1}{q-2k}\right).
\]
Since $\frac{(q-1)/2}q < \frac12$, and $0<\frac1{q-2k} < \frac12$ if $k\leq 0$, $\tW(\widetilde\alpha_k)$ admits a horizontal foliation by \eqref{condition:-1} of Theorem~\ref{thm:brwcondition}.

Second, we have $\ttau = q \tgamma + (-1-q) \tphi$ and $\tsigma = \tgamma - \tphi$.  Then, $\tW(\alpha_k)$ has Seifert invariants 
\[
M\left(-1,\frac{(q-1)/2}{q},\frac{(q-1)/2}{q},\frac{-1}{2k+q}\right).
\]
Since $\frac{(q-1)/2}q < \frac12$, and $0<\frac{-1}{2k+q} < \frac12$ if $k \ll 0$, $\tW(\widetilde\alpha_k)$ admits a horizontal foliation by \eqref{condition:-1} of Theorem~\ref{thm:brwcondition}.  
\end{proof}

We next consider the case $\gcd(n,q) \ne 1$, i.e. $\widetilde J$ is not connected.

\begin{proposition}\label{prop:234} 
Let $K$ be a non-trivial knot in $S^3$. 
Then the manifolds $\Sigma_2 (C_{3,4}(K))$, $\Sigma_4(C_{3,2}(K))$, and $\Sigma_3 (C_{2,3}(K))$ are excellent.
\end{proposition}

\begin{proof} 
We use the same arguments as in Proposition~\ref{prop:exceptional}.  As before, it suffices to show that $\widetilde{W}(\widetilde\alpha_k,\ldots,\widetilde\alpha_k)$ admits a horizontal foliation.  Recall that $\ttau = \omega\tgamma + \eta \tphi$, where $\omega = \frac{nq}{\gcd(n,q)\gcd(n,p)}$, and we solve for the possible values of $\eta$ using $\ttau \cdot \tsigma = 1$ .  \\

\noindent {\bf Case 1:} $(n,p,q) = (2,3,4)$. \\
In this case $\widetilde J$ has two components, each mapping homeomorphically to $J$.  By \cite[Theorem~1]{NR}, $\Sigma_2 (T_{3,4}) = M(\frac12, -\frac13, -\frac13)$.   
Since $J$ has multiplicity~3, each component of $\widetilde J$ is a fiber of multiplicity~3.  
Thus, $\widetilde\sigma_i = \pm(-3\widetilde\gamma_i + \widetilde\varphi_i)$. In this case, there is only one choice for $\eta$, and we find $\tsigma_i = -3\tgamma_i + \tphi_i$ and $\ttau_i = 4\tgamma_i - \widetilde\varphi_i$, for $i = 0,1$.

We have $\Sigma_2 (C_{3,4}(K)) \cong X_0 \cup X_1 \cup \widetilde W$, where  $X_i$ is a copy of $X$, glued along the two boundary components of $\widetilde W$. 
On $\partial X_i$ the slope $\alpha_k = \mu_i +k\lambda_i$ is identified with $\widetilde\tau_i + k\widetilde\sigma_i = (4-3k)\widetilde\gamma_i + (k-1)\widetilde\varphi_i$. 
Hence $\widetilde W(\alpha_k, \alpha_k)$ is the Seifert fibered space $M(\frac12, \frac{k-1}{4-3k}, \frac{k-1}{4-3k})  = -M( -1,\frac12, \frac{k-1}{3k-4}, \frac{k-1}{3k-4})$.
Since $0<\frac{k-1}{3k-4} < \frac13$ if $k\le 0$, $\widetilde W(\alpha_k,\alpha_k)$ has a horizontal foliation by Condition \eqref{condition:-(n-1)} of Theorem~\ref{thm:brwcondition}.  \\

\noindent {\bf Case 2:} $(n,p,q) = (4,3,2)$. \\
By \cite[Theorem 1]{NR}, $\Sigma_4 (T_{3,2}) = M(-\frac12, -\frac13, -\frac13,\frac11)$. The curve $J$ has multiplicity~3, $\widetilde J$ has two components, each mapping to $J$ by a covering map of degree~2, and each being a fiber of multiplicity~3. 
So, $\widetilde\sigma_i = \pm(-3\widetilde\gamma_i + \widetilde \varphi_i)$.  There is only one choice for $\eta$, and we see $\widetilde\tau_i = 4\widetilde\gamma_i - \widetilde\varphi_i$ and $\tsigma_i = -3\tgamma_i + \tphi_i$.  

The manifold $\Sigma_4 (C_{3,2}(K)) \cong X^{(2)}_0 \cup X^{(2)}_1 \cup \widetilde W$, where  each $X^{(2)}_i$ is a copy of $X^{(2)}$. 
The slope $\widetilde\alpha_k = \widetilde\mu_i + 2k\widetilde\lambda_i$ on $\partial X_i^{(2)}$ is identified with $\widetilde\tau_i + 2k\widetilde\sigma_i = (4-6k)\widetilde\gamma_i + (2k-1)\widetilde\varphi_i $. 
So $\widetilde W (\widetilde\alpha_k,\widetilde\alpha_k)$ is the Seifert manifold $M(1,-\frac12, \frac{2k-1}{4-6k},\frac{2k-1}{4-6k}) = -M(-1,\frac12,\frac{2k-1}{6k-4},\frac{2k-1}{6k-4})$.   
For $k\le 0$, we have $\frac{2k-1}{6k-4} < \frac13$, so the result follows from Condition~\eqref{condition:-(n-1)} of Theorem~\ref{thm:brwcondition}. \\

\noindent {\bf Case 3:} $(n,p,q) = (3,2,3)$. \\
By \cite[Theorem~1]{NR}, $\Sigma_3 (T_{2,3}) = M(-\frac12, -\frac12, -\frac12, \frac11)$.  In this case, 
$J$ has multiplicity~2, $\widetilde J$ has three components, and each component of $\widetilde J$ has multiplicity~2. 
Therefore, $\widetilde\sigma_i = \pm (-2\widetilde\gamma_i + \widetilde\varphi_i)$.  There are two cases.  

First, $\eta = -1$.  We have $\ttau_i = 3 \tgamma_i - \tphi_i$, $\tsigma_i = -2\tgamma_i + \tphi_i$.  We have $\Sigma_3 (C_{2,3}(K)) \cong X_0\cup X_1 \cup X_2 \cup \widetilde W$, where each $X_i$ is a copy of $X$.  The gluing of the boundaries identifies $\alpha_k = \mu_i + k\lambda_i$ with $\widetilde\tau_i + k\widetilde\sigma_i = (3-2k)\widetilde\gamma_i + (k-1)\widetilde\varphi_i$.  Therefore $\widetilde W (\alpha_k,\alpha_k,\alpha_k)$ is the Seifert fibered space $M(1, \frac{k-1}{3-2k}, \frac{k-1}{3-2k}, \frac{k-1}{3-2k}) = -M(-1,\frac{k-1}{2k-3},\frac{k-1}{2k-3},\frac{k-1}{2k-3})$.   
Since $0<\frac{k-1}{2k-3} < \frac12$ if $k\leq 0$, Condition \eqref{condition:-(n-1)} of Theorem~\ref{thm:brwcondition} implies that $\widetilde W$ has a horizontal foliation.

Second, $\eta = -2$.  We have $\ttau_i = 3 \tgamma_i - 2\tphi_i$, $\tsigma_i = 2\tgamma_i - \tphi_i$.  Therefore $\widetilde W (\alpha_k,\alpha_k,\alpha_k)$ is the Seifert fibered space $-M(-1, \frac{k+2}{2k+3}, \frac{k+2}{2k+3}, \frac{k+2}{2k+3})$. 
Since $0<\frac{k+2}{2k+3} < \frac12$ if $k \ll 0$, Condition \eqref{condition:-1} of Theorem~\ref{thm:brwcondition} implies that $\widetilde W$ has a horizontal foliation.
\end{proof}

The following completes the proof of Theorem~\ref{thm:cables}.

\begin{proposition}\label{prop:1q}
Let $K$ be a non-trivial knot in $S^3$. 
Unless $n=q=2$, $\Sigma_n (C_{1,q}(K))$ is excellent.
\end{proposition}

\begin{proof}
Let $d = \gcd(n,q)$ and let $r = \frac{n}{d}$ and $s = \frac{q}{d}$.  Note that $\gcd(r,s) = 1$.  From the discussion immediately preceding Proposition~\ref{prop:generic}, we have that $\Sigma_n(C_{1,q}(K)) = \widetilde{W} \cup \coprod^d_{i=1} X^{(r)}_i $, where 
\[
\widetilde{W} = \Sigma_n(T_{1,q}) - \coprod^d_{i=1} int(\widetilde{V}_i) = S^3 - int(N(\widetilde{J})).
\]
To identify $\widetilde{J} \subset S^3$, it is convenient to note that there is an isotopy of $S^3$ that interchanges the components of $C$ and $J$.  With this picture in mind, it is then clear that $\widetilde{J}$ is the $(n,q)$-torus link $T_{n,q} = T_{dr,ds}$.  Let $\widetilde{J}_1,\ldots,\widetilde{J}_d$ denote the components of $\widetilde{J}$.  

As before, let $\widetilde{\sigma}_i$ be the meridian of $\widetilde{V}_i$ and $\widetilde \tau_i \subset \partial \widetilde{V}_i$, the lift of $\tau$.  Let $\widetilde{\rho}_i \subset \partial \widetilde{V}_i$ be the 0-framed longitude of $\widetilde{J}_i \subset \Sigma_n(T_{1,q}) \cong S^3$.  Since $\widetilde \tau_i \cdot \widetilde{\sigma}_i = 1$ on $\partial \widetilde{W}$, we have $\widetilde \tau_i = \widetilde \rho_i + c \widetilde \sigma_i$ for some $c \in \mathbb{Z}$.  Let $\alpha_k$ be the slope $\mu + k\lambda$ on $\partial X$, and $\widetilde{\alpha}_k = \widetilde \mu_i + rk\widetilde \lambda_i$ the corresponding slope on $\partial X^{(r)}_i$.  Under the gluing homeomorphism $\partial X^{(r)}_i \to \partial \widetilde{V}_i \subset \partial \widetilde W$, we have that $\widetilde{\alpha}_k$ is mapped to $\widetilde \tau_i + rk \widetilde \sigma_i = \widetilde \rho_i + (c + rk) \widetilde \sigma_i$.  By Proposition~\ref{prop:negtorus}, for $k$ sufficiently negative, the multislope $(\widetilde{\alpha}_k,\ldots,\widetilde{\alpha}_k)$ is excellent for $\widetilde{W}$, unless $n = q = 2$.  By Lemma~\ref{lem:cover1/n}, we have that for $k \ll 0$, $\widetilde{\alpha}_k$ is an excellent slope for for each $X^{(r)}_i$.  The result again follows from Lemma~\ref{lem:excellentgluing}.     
\end{proof}

\subsection{Cyclic branched covers of $(p,2)$-cables}
We now prove Theorems~\ref{thm:conditional} and \ref{thm:cablefail}.  Let $p \geq 1$ be odd.  In this case, the link $C \cup J$ from the discussion just before Proposition~\ref{prop:generic} is illustrated in Figures~\ref{fig:cables2fig1} and ~\ref{fig:cables2fig2}, which show the cases $p = 5$ and $p = 1$ respectively.  For the special case of $p = 1$, we relabel the components as $C' \cup J'$, with associated meridian-longitude pair $\mu', \lambda'$.  

\begin{figure}
\labellist 
\large
\pinlabel  $C$ at 98 167
\pinlabel  $J$ at 130 77
\endlabellist
\includegraphics[scale=.7]{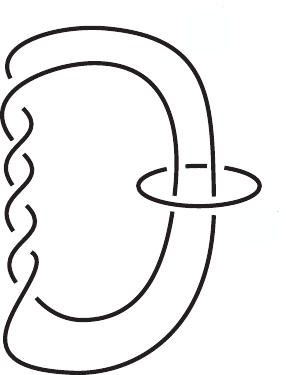}
\caption{The union of a regular fiber, $C$, and an exceptional fiber, $J$, of multiplicity $p = 5$, in a Seifert fibration of $S^3$.}\label{fig:cables2fig1}
\end{figure}
\begin{figure}
\labellist 
\large
\pinlabel  $C$ at 81 110
\pinlabel  $J$ at 114 42
\endlabellist
\includegraphics[scale=.8]{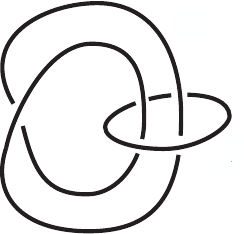}
\caption{The union of a regular fiber, $C$, and an exceptional fiber, $J$, of multiplicity $p = 1$ in a Seifert fibration of $S^3$.}\label{fig:cables2fig2}
\end{figure}

Now, for any $p$, the exterior of $C \cup J$ is in fact homeomorphic to $C' \cup J'$.  One way to see this is as follows.  The exterior of $C \cup J$ is the exterior of a regular fiber in the $\frac{2 \pi p}{2}$ Seifert fibered solid torus.  Therefore, applying $r = \frac{p-1}{2}$ negative meridional Dehn twists to this solid torus produces a $\frac{2 \pi 1}{2}$ Seifert fibered solid torus; this is the exterior of $C' \cup J'$.  In other words, there is a homeomorphism $h: \overline{W - N(C)} \to \overline{W' - N(C')}$ such that $h(\tau) = \tau'$ and $h(\sigma) = \sigma' - r \tau'$.  

We may interchange the components of $C' \cup J'$ to obtain the link in Figure~\ref{fig:cables2fig3}.  
\begin{figure}
\includegraphics[scale=.8]{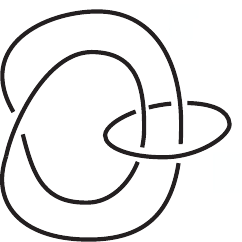}
\labellist 
\large
\pinlabel  $J'$ at 90 338
\pinlabel  $C'$ at 135 250
\endlabellist
\caption{The result of interchanging the components of $C' \cup J'$.}\label{fig:cables2fig3}
\end{figure}
From this, we can see that $\widetilde{J}' \subset \Sigma_2(T_{1,2}) \cong S^3$ is the Hopf link (see Figure~\ref{fig:cables2fighopf}).  

\begin{figure}
\includegraphics[scale=.65]{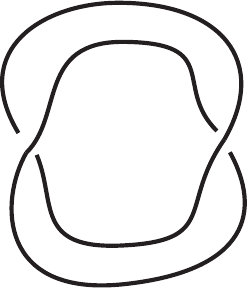}
\labellist 
\large
\pinlabel  $J$ at 85 335
\pinlabel  $C$ at 135 250
\endlabellist
\caption{The lift, $\widetilde{J}'$, of $J'$ in $\Sigma_2(T_{1,2})$.}\label{fig:cables2fighopf}
\end{figure}
Let the components of $\widetilde{J}'$ be $J'_0$ and $J'_1$, and let the corresponding lifts of $V'$ be $V'_0$ and $V'_1$.  Then $\widetilde W' \cong \mathbb{T}^2 \times I$, where $\partial V'_i = \mathbb{T}^2 \times \{i\}$, for $i = 0,1$.  Consequently, we have the corresponding decomposition $\Sigma_2(T_{p,2}) = V_0 \cup V_1 \cup \widetilde{W}$.  We would like to determine the gluing.  If $\alpha$ is a slope on $\partial V$ (respectively $\partial V'$), let $\alpha_i$ (respectively $\alpha'_i$) be the corresponding slope on $\partial V_i$ (respectively $\partial V'_i$), for $i = 0,1$.  Note that the homeomorphism $h$ lifts to a homeomorphism $\widetilde{h}:\widetilde W \to \widetilde W'$ such that $\widetilde h(\tau_i) = \tau_i$ and $\widetilde h(\sigma_i) = \sigma'_i - r\tau'_i$, for $i = 0,1$.  

Let $\beta'$ be the blackboard framed longitude of $J'$ in Figure~\ref{fig:cables2fig3}.  Then $\beta' = \tau' + \sigma'$, and $\beta'_i = \tau'_i + \sigma'_i$ is the 0-framed longitude of $J'_i$.  

Recall that $\Sigma_2(C_{p,2}(K)) \cong X_0 \cup X_1 \cup \widetilde{W}$, where $X_i$ is the copy of the exterior $X$ of $K$, for $i = 0,1$, and the gluing homeomorphism $\partial X_i \to \partial V_i$ takes $\lambda_i$ to $\sigma_i$ and $\mu_i$ to $\tau_i$, for $i = 0,1$.  Using the homeomorphism $\widetilde{h}$, $\Sigma_2(C_{p,2}(K)) \cong X_0 \cup X_1 \cup \widetilde{W}'$, where the gluing homeomorphism $\partial X_i \to \partial V'_i = \mathbb{T}^2 \times \{i\}$ now takes $\lambda_i$ to 
\[
\sigma'_i - r\tau'_i = \sigma'_i - r(\beta'_i - \sigma'_i) = (r+1) \sigma'_i - r\beta'_i,
\]
and takes $\mu_i$ to $\tau'_i = - \sigma'_i + \beta'_i$, for $i = 0,1$.

Thus, with respect to the ordered bases $\lambda_i, \mu_i$ and $\sigma'_i,\beta'_i$, this gluing homeomorphism is given by the matrix $A = \begin{pmatrix} r+1 & -1 \\ -r & 1 \end{pmatrix}$.  Since $\widetilde{W}' \cong \mathbb{T}^2 \times I$, with $\sigma'_0, \beta'_0$ homotopic to $\beta'_1, \sigma'_1$, respectively, it follows that $\Sigma_2(C_{p,2}(K)) \cong X_0 \cup X_1$, glued by the homeomorphism $f: \partial X_0 \to \partial X_1$ that is given with respect to the ordered bases $\lambda_0,\mu_0$ and $\lambda_1,\mu_1$, by the matrix $A^{-1}BA$, where $B = \begin{pmatrix} 0 & 1 \\ 1 & 0\end{pmatrix}$, i.e. by $\begin{pmatrix} 1 & 0 \\ 2r+1 & -1 \end{pmatrix} = \begin{pmatrix} 1 & 0 \\ p & -1 \end{pmatrix}$.  In other words, we have 
\begin{equation}\label{eqn:cp2}
\Sigma_2(C_{p,2}(K))  = X_0 \cup_f X_1, \; \; f(a \mu_0 + b \lambda_0) =  (pb-a)\mu_1 + b \lambda_1.  
\end{equation}
Observe that under this identification, a slope $\frac{a}{b}$ on $\partial X_0$ is sent to the slope $p - \frac{a}{b}$ on $\partial X_1$.

\begin{proof}[Proof of Theorem~\ref{thm:conditional}]
From \eqref{eqn:cp2}, we have $\Sigma_2(C_{1,2}(K)) = X_0 \cup X_1$, where the slope $\frac{a}{b}$ is identified with $1-\frac{a}{b}$.  The slopes $\frac{1}{2}$ on $\partial X_0$ and $\partial X_1$ are therefore identified.  If $\frac{1}{2}$ is a CTF slope for $K$, then $\frac{1}{2}$ is an excellent slope by Lemma~\ref{lem:toroidalzs3}.  Thus, $\Sigma_2(C_{1,2}(K))$ is excellent by Lemma~\ref{lem:excellentgluing}.
\end{proof}

\begin{proof}[Proof of Theorem~\ref{thm:cablefail}]
By \eqref{eqn:cp2}, $\Sigma_2(C_{p,2}(K)) \cong X_{0} \cup X_{1}$, where $X_0$ and $X_1$ are copies of the exterior $X$ of $K$, glued along their boundaries so that the slope $\frac{a}{b}$ on $\partial X_0$ is identified with the slope $p - \frac{a}{b}$ on $\partial X_1$. In particular, the meridians of $X_0$ and $X_1$ are identified. 

Let $\mathcal{L}(X)$ be the set of L-space slopes on $\partial X$, i.e. $\mathcal{L}(X) = \{\alpha \mid X(\alpha)\text{ is an L-space}\}$. By \cite{KMOS,HFKQ},
$$\mathcal{L}(X) = \{\infty \}, \quad [2g-1,\infty], \quad \textup{or} \quad [- \infty, 1-2g],$$ 
where $g$ is the genus of $K$. Here, the intervals are to be interpreted as being in $\mathbb{Q} \cup \{\infty \}$. Hence the meridian $\mu = \infty$ of $K$ is not in the interior $\mathcal{L}^{o}(X)$. Therefore, by \cite{HRW}, $\Sigma_2(C_{p,2}(K))$ is not an L-space.

If $K$ is a torus knot or iterated torus knot then $\Sigma_2(C_{p,2}(K))$ is a graph manifold. By \cite{BC} and \cite{HRRW}, for graph manifolds the properties of not being an L-space, having a co-orientable taut foliation, and having left-orderable fundamental group, are equivalent. This proves the second part of the theorem.  
\end{proof}

\begin{remark}\label{rmk:correction}
The mistake in the previous version of Theorem~\ref{thm:cablefail} occurs in the last three sentences of the original argument.  There it was incorrectly assumed that $\infty$ is not a foliation-detected slope; indeed, $\infty$ is a foliation-detected slope for the trefoil \cite[Corollary A.7]{BC}.
\end{remark}

\section{Whitehead doubles}\label{sec:whitehead}

Let $K$ be a knot in $S^3$ and let $Wh(K)$  be the {\em positive untwisted Whitehead double} of $K$.  This can be described as follows.  Consider the Whitehead link with a positive clasp, $C \cup J$, as shown in Figure~\ref{fig:whiteheadfigure1}.  
\begin{figure}
\labellist
\large
\pinlabel $C$ at 80 5
\pinlabel $J$ at 175 85
\endlabellist
\includegraphics[scale=.7]{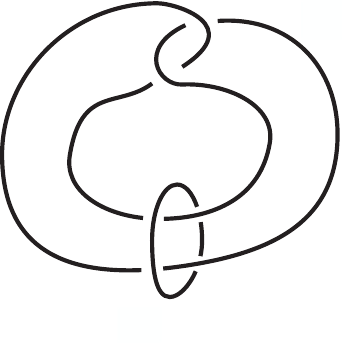}
\caption{The Whitehead link $C \cup J$.}\label{fig:whiteheadfigure1}
\end{figure}

Let $N(C)$  be a tubular neighborhood of $C$, disjoint from $J$, with meridian $m$ and 0-framed longitude $l$.  Let $X$ be the exterior of $K$, with meridian $\mu$ and longitude $\lambda$.  Remove $int N(C)$ from $S^3$ and replace it with $X$, identifying the torus boundaries in such a way that $m$ is identified with $\lambda$ and $l$ is identified with $\mu$.  In the resulting $S^3$, the image of $J$ is $Wh(K)$.  Observe that if one considers the negative untwisted double, where we instead use the Whitehead link with a negative clasp, we obtain the mirror of the positive untwisted Whitehead double of the mirror of $K$.  For this reason, we restrict our attention to positive untwisted Whitehead doubles.

To prove Theorem~\ref{thm:whitehead} we first give an explicit description of $\Sigma_n(Wh(K))$.  Since the components of the Whitehead link are interchangeable, we can redraw the link as in Figure~\ref{fig:whiteheadfigure2}.

\begin{figure}
\labellist
\large
\pinlabel $J$ at 80 5
\pinlabel $C$ at 175 85
\endlabellist
\includegraphics[scale=.7]{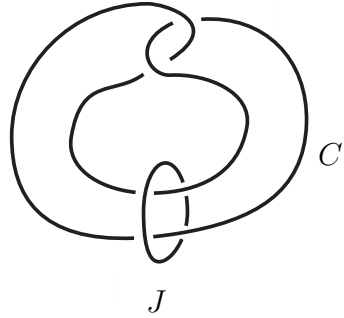}
\caption{The Whitehead link after interchanging the components in Figure~\ref{fig:whiteheadfigure1}.}\label{fig:whiteheadfigure2}
\end{figure}

The $n$-fold cyclic branched cover of $(S^3,J)$ is $S^3$, and the inverse image of $C$ under the covering projection is the $n$-component chain link $L_n$ shown in Figures~\ref{fig:whiteheadfigure3} and \ref{fig:whiteheadfigure4}, which illustrate the cases $n = 2$ and $n = 5$ respectively.  
\begin{figure}
\labellist
\large
\pinlabel $C_0$ at 0 55
\pinlabel $C_1$ at 205 115
\endlabellist
\includegraphics[scale=.6]{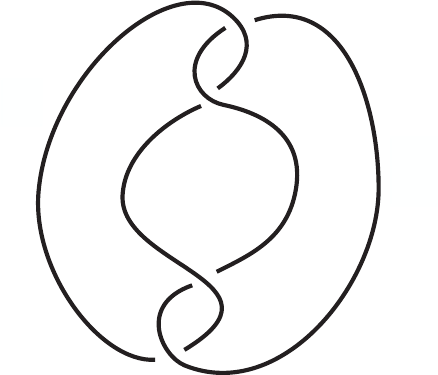}
\caption{The 2-component chain link $L_2$.}\label{fig:whiteheadfigure3}
\end{figure}
\begin{figure}
\labellist
\large
\pinlabel $C_4$ at 20 80
\pinlabel $C_2$ at 305 115
\pinlabel $C_1$ at 265 260
\pinlabel $C_0$ at 40 250
\pinlabel $C_3$ at 185 15
\endlabellist
\includegraphics[scale=.55]{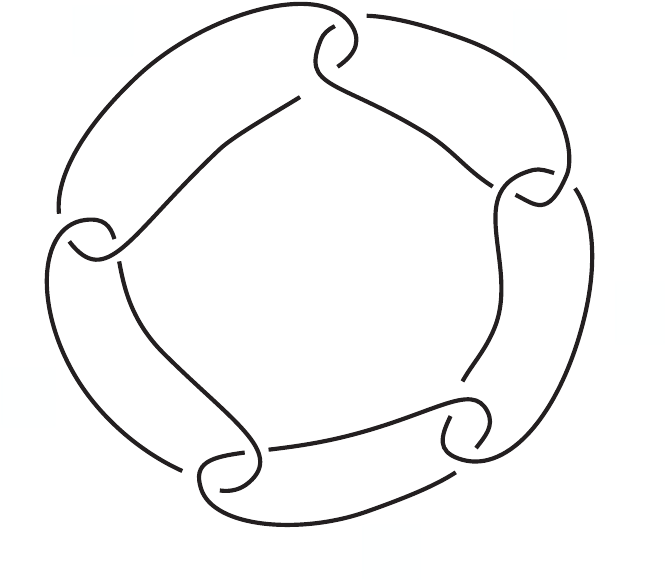}
\caption{The 5-component chain link $L_5$.}\label{fig:whiteheadfigure4}
\end{figure}
Since $\ell k (J,C) = 0$, we have that $\pi^{-1}(C)$ has $n$ components, $C_i$ for $i \in \Z/n$.  Let $m_i, l_i$ on $\partial N(C_i)$ be the lifts of $m,l$ on $\partial N(C)$.  Let $Y_n = S^3 - \coprod_{i \in \Z/n} int N(C_i)$ be the exterior of $L_n$.  Then, because $\ell k(C,J) = 0$, we have that $\Sigma_n(Wh(K))$ is obtained by gluing to $Y_n$, for each $i$, a copy $X_i$ of $X$ along $\partial N(C_i)$ in such a way that $m_i$ is sent to $\lambda_i$ and $l_i$ is sent to $\mu_i$, where $\mu_i, \lambda_i$ are the corresponding copies of $\mu,\lambda$ on $\partial X_i$.    

Let $b$ be the blackboard framing of $C$ corresponding to Figure~\ref{fig:whiteheadfigure2}; see Figure~\ref{fig:whiteheadfigure5}.  
\begin{figure}
\labellist
\Large
\pinlabel $C$ at 193 123
\pinlabel $b$ at 152 100
\endlabellist
\includegraphics[scale=.8]{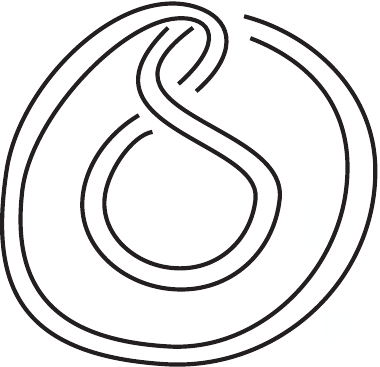}
\caption{The blackboard framing of $C$ from Figure~\ref{fig:whiteheadfigure2}.}\label{fig:whiteheadfigure5}
\end{figure}
Then, in $H_1(\partial N(C))$, we have $b = 2m + l$.  The lift $b_i$ of $b$ in $\partial N(C_i)$ is the 0-framed longitude of $C_i$ (see Figure~\ref{fig:whiteheadfigure6} for the case $n = 2$), and $l_i = -2m_i + b_i$.  
\begin{figure}
\labellist
\Large
\pinlabel $C_1$ at 204 125
\pinlabel $b_1$ at 158 110
\pinlabel $C_0$ at -10 45
\pinlabel $b_0$ at 35 60
\endlabellist
\includegraphics[scale=.75]{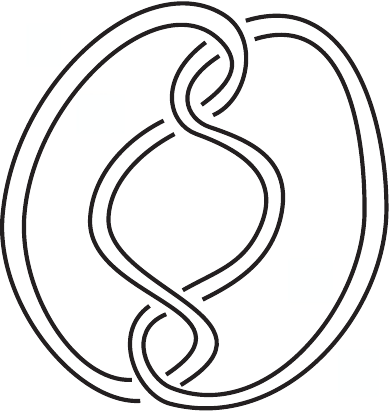}
\caption{The lifts of the blackboard framing of $C$ from Figure~\ref{fig:whiteheadfigure2}.}\label{fig:whiteheadfigure6}
\end{figure}
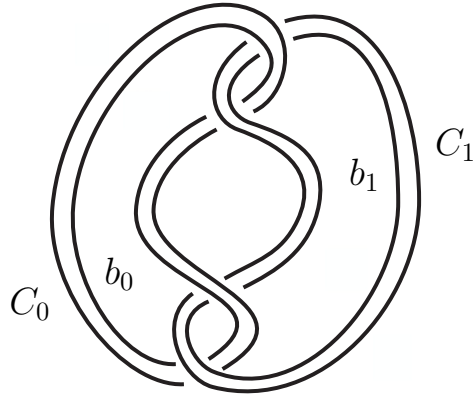

\begin{proof}[Proof of Theorem~\ref{thm:whitehead}]
We first consider the case of two-fold branched covers.  Note that the link $L_2$ is the negative torus link $T_{-2,4}$.  Consider the slope $\mu_i + n_i \lambda_i$ on $X_i$.  Note that this slope corresponds to $\frac{1}{n_i}$ on $K$.  For $|n_i| \gg 0$ this is an excellent slope by Lemma~\ref{lem:1/n}.  Further, by the above discussion, the slope $\mu_i + n_i\lambda_i$ is identified with the slope $(n_i - 2)m_i + b_i$.  Since $b_i$ represents the longitude on the $i$th component of $T_{-2,4}$, we have that the slopes $\mu_1 + n_1 \lambda_1$ and $\mu_2 + n_2 \lambda_2$ on $X_1$ and $X_2$ respectively are identified with the multislope $(\frac{n_1-2}{1},\frac{n_2-2}{1})$ on the exterior of $T_{-2,4}$.  By Proposition~\ref{prop:negtorus}, we have that for $n_1,n_2 \gg 0$, $(\frac{2-n_1}{1}, \frac{2 - n_2}{1})$ is an excellent multislope on the exterior of $T_{2,4}$.  Therefore, we have that $(\frac{n_1-2}{1},\frac{n_2-2}{1})$ is an excellent multislope on the exterior of $T_{-2,4}$ for $n_1,n_2 \gg 0$.  It now follows from Lemma~\ref{lem:excellentgluing} that $\Sigma_2(Wh(K))$ is an excellent manifold.  Since $\Sigma_{2n}(Wh(K))$ is a branched cover of $\Sigma_2(Wh(K))$, Theorem~\ref{thm:brw} shows that $\pi_1(\Sigma_{2n}(Wh(K))$ is left-orderable for all $n$.  

We now study the higher-order branched covers.  There is a degree one map from $X$ to $S^1 \times D^2$ which restricts to a homeomorphism $\partial X \to \partial(S^1 \times D^2)$ mapping $\lambda$ to a meridian of $S^1 \times D^2$ and $\mu$ to a longitude of $S^1 \times D^2$.  Applying this map to each $X_i$ except $X_0$ defines a degree 1 map $f: \Sigma_n(Wh(K)) \to M$, where $M$ is obtained from $Y_n$ by filling in each $N(C_i)$ for $i \neq 0$, while leaving $X_0$ attached along $\partial N(C_0)$ as before.  

Now, $S^3 - int N(C_0)$ is a solid torus whose meridian is $b_0$.  Since $b_0 = 2m_0 + l_0$ is identified with $2 \lambda_0 + \mu_0$ in $\partial X_0$, $M$ is homeomorphic to $S^3_{\frac{1}{2}}(K)$.  Suppose first that $\pi_1(S^3_{\frac{1}{2}}(K))$ is left-orderable.  Since $f$ has degree one and $\Sigma_n(Wh(K))$ is irreducible, we have by Theorem~\ref{thm:brw} that $\pi_1(\Sigma_n(Wh(K)))$ is left-orderable.  

Now assume that $n \geq 3$ and apply the procedure described above to each $X_i$ except $X_0$ and $X_1$.  This gives a degree 1 map from $\Sigma_n(Wh(K))$ to $Q$, the manifold obtained by attaching $X_0$ and $X_1$ to $E = S^3 - (int (N(C_0)) \cup int (N(C_1)))$, the exterior of the Hopf link, which is homeomorphic to $\mathbb{T}^2 \times I$.  In other words, $Q \cong X_0 \cup_h X_1$, for some gluing homeomorphism $h: \partial X_0 \to \partial X_1$.  To determine $h_* :H_1(\partial X_0) \to H_1(\partial X_1)$, we note that the homeomorphism $\partial N(C_0) \to \partial N(C_1)$ given by the product structure on $E$ sends $m_0$ to $b_1$ and $b_0$ to $m_1$, since the $b_i$ are longitudes for $C_i$.  Then, we have that 
\[
h_*: \mu_0 \mapsto l_0 = -2m_0 + b_0 \mapsto -2b_1 + m_1 \mapsto -2(\mu_1 + 2\lambda_1) + \lambda_1 = -2\mu_1 -3\lambda_1
\]
and
\[
\lambda_0 \mapsto m_0 \mapsto b_1 \mapsto \mu_1 + 2\lambda_1.  
\]  

Thus, with respect to the ordered bases $\mu_0, \lambda_0$ and $\mu_1,\lambda_1$, $h_*$ is given by the matrix $\begin{pmatrix} -2 & 1 \\ -3 & 2 \end{pmatrix}$.  

Let $\alpha_0$ be the slope $\frac{1}{k} = \mu_0 + k\lambda_0$ on $\partial X_0$.  Then, $h_*(\alpha_0) = (k-2)\mu_1 + (2k-3)\lambda_1$ on $\partial X_1$.  This will be of the form $\frac{1}{k'}$ if and only if $k = 1$ or $k=3$, in which cases we have $h_*(\frac{1}{1}) = \frac{1}{1}$ and $h_*(\frac{1}{3}) = \frac{1}{3}$.  

Observe that $Q$ is irreducible.  It follows from Theorem~\ref{thm:jsjlo} that if $\pi_1(S^3_{1}(K))$ or $\pi_1(S^3_{\frac{1}{3}}(K))$ is left-orderable, then so is $\pi_1(Q)$. Therefore, since we have a degree one map from $\Sigma_n(Wh(K))$ onto $Q$, we may again apply Theorem~\ref{thm:brw} to conclude that $\pi_1(\Sigma_n(Wh(K))$ is left-orderable.  
\end{proof}

We have the following obvious corollary.  
\begin{corollary}\label{cor:higherwhite}
Let $K$ be a knot in $S^3$ that satisfies the Li-Roberts Conjecture.  Then $\pi_1(\Sigma_n(Wh(K))$ is left-orderable for all $n \geq 2$.  
\end{corollary}

\section{Some excellent cyclic branched covers of two-bridge knots}\label{sec:twobridgesurgery}

Consider the two-bridge knots corresponding to rational numbers of the form 
\[
\frac{2(2k+1)(2\ell +1) +\ep}{(2\ell +1)} =  [2(2k+1),\ep (2\ell +1)],
\]
where $k,\ell \ge 1$ and $\ep = \pm 1$.  By \cite[Corollary 11.2(b)]{Minkus1982},  $\Sigma_{2k+1} (K_{[2(2k+1),\ep (2\ell +1)]})$ is an integer homology sphere.  By Lemma~\ref{lem:pi1surjective}, $\Sigma_n (K_{[2(2k+1), \ep (2\ell+1)]})$ is also an integer homology sphere if $n$ divides $(2k+1)$.  We are interested in the case of $\ep = +1$ for Theorem~\ref{thm:twobridgesurgery}.    

Before proving Theorem~\ref{thm:twobridgesurgery}, we give a corollary which easily follows from the theorem.  
\begin{corollary}\label{cor:zs3}
If $\gcd(n,2k+1) >1$ then $\pi_1 (\Sigma_n (K_{[2(2k+1), 2\ell+1]}))$ is left-orderable. 
\end{corollary}
\begin{proof}
Let $m = \gcd(n,2k+1)$.  Then, we have a non-zero degree map from $\Sigma_n (K_{[2(2k+1), 2\ell+1]})$ to $\Sigma_m(K_{[2(2k+1), 2\ell+1]})$.  Theorem~\ref{thm:twobridgesurgery} applies to $\Sigma_m(K_{[2(2k+1), 2\ell+1]})$ and the result now follows from Theorem~\ref{thm:brw}.
\end{proof}

Let $B(n,k,\ell,\text{sign}(\ep))$ denote the manifold $\Sigma_n (K_{[2(2k+1),\ep (2\ell+1)]})$.  Theorem~\ref{thm:twobridgesurgery} will be proved by studying the following surgery description of $B(n,k,\ell,\pm)$ when $n$ divides $(2k+1)$. 

\begin{lemma}\label{lem:pretzelsurgery}
Suppose $n$ divides $2k+1$, $n>1$, and let $d = (2k+1)/n$.  Let $P(2\ell+1,\ldots,2\ell+1)$ be the $n$-stranded pretzel knot with $(2\ell+1)$ right-handed half-twists in each strand.  Then 
\begin{itemize}
\item[(1)] $B(n,k,\ell,+) \cong -S^3_{\frac{1}{d}}(P(2\ell+1,\ldots,2\ell+1))$,
\item[(2)] $B(n,k,\ell,-) \cong S^3_{\frac{-1}{d}}(P (2\ell+1,\ldots, 2\ell+1))$.
\end{itemize}
\end{lemma}

\begin{proof} 
(1) Let $K = - K_{[2(2k+1),2\ell+1]}$; see Figure~\ref{fig:figuretwobridge1}, which illustrates the case $k=1$, $\ell=2$.
\begin{figure}
\includegraphics[scale=.65]{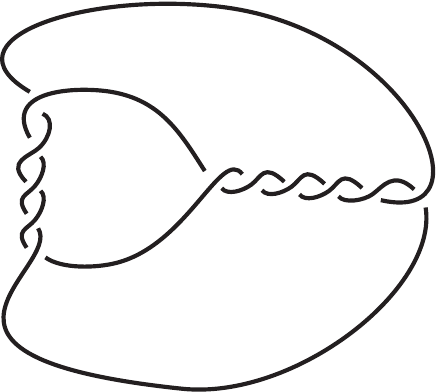}
\caption{The knot $-K_{[6,5]}$.}\label{fig:figuretwobridge1}
\end{figure}
Consider the 2-component link $J\cup C$ shown in Figure~\ref{fig:figuretwobridge2}.
\begin{figure}
\labellist
\Large
\pinlabel $J$ at 40 100
\pinlabel $C$ at 162 100
\endlabellist
\includegraphics[scale=.8]{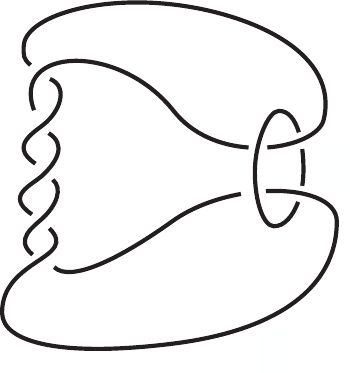}
\caption{The 2-component link $J \cup C$.  We can obtain $K$ from $J$ by doing $\frac{1}{2k+1}$-surgery on $C$.}\label{fig:figuretwobridge2}
\end{figure}
Performing $\frac{1}{2k+1}$-surgery on $C$ transforms $(S^3,J)$ to $(S^3,K)$.  Since $J$ is unknotted, the $n$-fold cyclic branched covering of $(S^3,J)$ is $S^3$; let $\widetilde C = \pi^{-1}(C)$ be the inverse image of $C$ under the branched covering projection~$\pi$.  Since $\ell k(C,J) = \pm2$,  and $n$ is odd, $\widetilde C$ is connected. In fact, since the components of $J\cup C$ are interchangeable, we can see from Figure~\ref{fig:figuretwobridge2} that $\widetilde C$ is the pretzel knot $P(2\ell +1,\ldots, 2\ell+1)$; see Figure~\ref{fig:figuretwobridge3}, which shows the case $k=1$, $\ell=2$, $n=3$.  
\begin{figure}
\includegraphics[scale=.5]{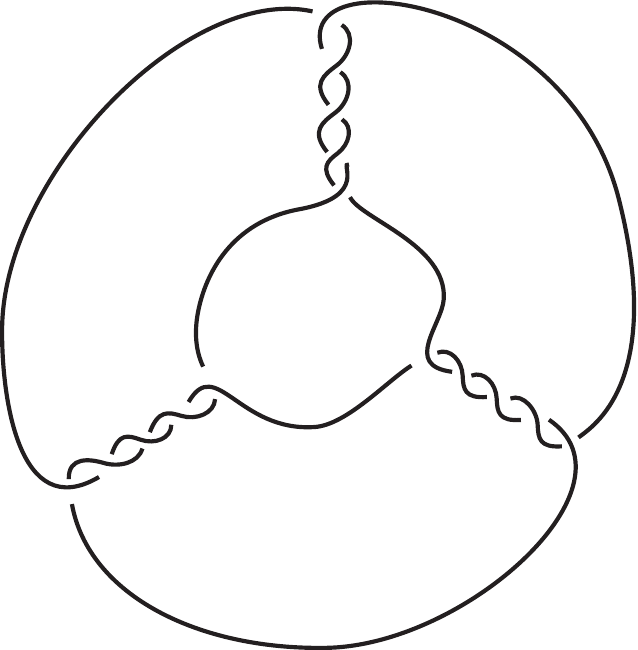}
\caption{The pretzel knot $P(5, 5, 5)$.}\label{fig:figuretwobridge3}
\end{figure}
Let $\mu,\lambda$ be a meridian and $0$-framed longitude of $C$ respectively.  We have that $\widetilde \lambda$ is connected, while since $\ell k(\mu,J) = 0$, we have $\pi^{-1}(\mu)$ consists of $n$ copies of $\widetilde\mu$, the meridian of $\widetilde C$.  Also, $C$ bounds a disk $D$ that meets $J$ in two points (see Figure~\ref{fig:figuretwobridge2} with the roles of $J$ and $C$ reversed).  Then $\pi^{-1}(D)$ is an orientable surface with boundary $\widetilde C$.  It follows that $\pi^{-1}(\lambda) = \widetilde\lambda$ is the $0$-framed longitude of $\widetilde C$.

Let $\alpha = \mu + (2k+1)\lambda$ be the slope $\frac{1}{2k+1}$ on the boundary of a neighborhood of $C$.  Since $n$ divides $2k+1$, $\pi^{-1}(\alpha)$ consists of $n$ copies of the slope $\widetilde\alpha = \widetilde\mu + d \widetilde\lambda$ on the boundary of a neighborhood of $\widetilde C$.  It follows that the $n$-fold cyclic branched covering of $K$ is obtained by $\frac{1}{d}$-surgery on $\widetilde C$, which proves (1).

The proof of (2) is completely analogous.
\end{proof}

\begin{proof}[Proof of Theorem~\ref{thm:twobridgesurgery}]
Since $P= P(2\ell+1,\ldots,2\ell+1)$ is alternating and negative, it follows from \cite[Theorem 0.2]{Roberts1995} that $S^3_r(P)$ has a taut foliation for all $r\in \mathbb{Q}$, $r\ge 0$. (Note that the convention for the signs of crossings in \cite{Roberts1995} is the opposite of the usual one.)  In particular $B(n,k,\ell,+)\cong -S^3_{\frac{1}{d}}(P)$ has a taut foliation.   Since $B(n,k,\ell,+)$ is an integer homology sphere, it is excellent by Lemma~\ref{lem:toroidalzs3}.
\end{proof}

\begin{remark}\label{rmk:nathan}
\begin{enumerate}
\item The manifolds $B(n,k,\ell,\pm)$ in Lemma~\ref{lem:pretzelsurgery} are not L-spaces by \cite[Theorem 1.5]{OSLens}.

\item\label{item:nathan} In the case $\ep =-1$, \cite[Theorem 0.2]{Roberts1995} does not apply to $S^3_{-\frac{1}{d}}(P(2\ell+1,\ldots, 2\ell+1))$; we do not know if $B(n,k,\ell,-)$ has a taut foliation, for any $k,\ell \ge 1$.

Nathan Dunfield has informed us that for both $k=\ell=1$ and $k=1$, $\ell=2$, $\pi_1 (B(2k+1,k,\ell,-))$ has a non-trivial representation into $PSL (2,\mathbb{R})$.  Since $B(2k+1,k,\ell,-)$ is an integer homology sphere, the obstruction to lifting to $\widetilde{SL_2(\mathbb{R})}$ vanishes.  Since $\widetilde{SL_2(\mathbb{R})}$ is a left-orderable group (it is a subgroup of $\operatorname{Homeo}_+(\mathbb{R})$), we may apply Theorem~\ref{thm:brw} to the lift.  Therefore, $\pi_1(B(2k+1,k,\ell,-))$ is left-orderable.  Dunfield has also shown that, in the case $\epsilon = +1$, $\pi_1(B(2k+1,k,\ell,+)$ has a non-trivial $PSL(2,\mathbb{R})$-representation for $k = 1$, $1 \leq \ell \leq 15$.

\item When $\ep =  +1$, we have $2(2k+1)(2\ell+1)+\ep \equiv 3 \pmod{4}$, and so \cite{Hu2013} implies that $\pi_1 (\Sigma_n (K_{[2(2k+1),2\ell+1]}))$ is left-orderable 
for all sufficiently large $n$.  We do not know if the corresponding manifolds have co-orientable taut foliations or whether or not they are L-spaces. 

Arguments similar to those that appear in \cite{Hu2013} are used in \cite{Tran2013} to show (in particular) that $\pi_1 (\Sigma_n (K_{[2(2k+1),- (2\ell +1)]})$ is also left-orderable for $n$ sufficiently large.  Moreover, in \cite{Tran2013} explicit lower bounds for $n$ are given, both for $\ep = +1$ and $\ep = -1$.  However, these bounds go to infinity as $\ell$ increases.  Again we do not know if these manifolds have co-orientable taut foliations, or whether or not they are L-spaces. 
\end{enumerate}
\end{remark}

\section{Total L-spaces arising from two-bridge knots}\label{sec:twobridge}

\begin{proof}[Proof of Theorem~\ref{thm:twobridge}]
That these manifolds are L-spaces is proved in \cite{T14}.  

It thus remains to show that $\pi_1(\Sigma_n(K_{[2\ell,-2k]}))$ is not left-orderable.  We first note that $[2\ell,-2k] = \frac{4k\ell-1}{2k} = [2\ell-1,1,2k-1]$.  We will use the presentation given in \cite[Proposition 3(d)]{DPT2005} for $\pi_1(\Sigma_n(K_{[2\ell-1,1,2k-1]}))$.  Replacing $k$ and $\ell$ in the presentation given there by $k-1$ and $\ell-1$, for each $i \in \Z/n$, the relators $r_i$ from this presentation can be written as 
\[
(x^{-k}_i x^k_{i+1})^{\ell} (x_{i+2}^{-k} x^k_{i+1})^{\ell-1}(x_{i+2}^{-k}x_{i+1}^{k-1}).
\]
Since there is an automorphism of $\pi_1(\Sigma_n(K_{[2\ell-1,1,2k-1]}))$ given by sending $x_i$ to $x_{i+1}$ for all $i \in \Z/n$, no $x_i$ is trivial.  Also, we see that in $r_i$, the occurrences of a given generator either all have positive exponent or all have negative exponent.  

We take $n = 4$ and write $a = x_0, b = x_1, c = x_2, d = x_3$.  Then, $\pi_1(\Sigma_4(K_{[2\ell,-2k]}))$ has a presentation with generators $a,b,c,d$ and relations 
\begin{align}
\label{eqn:rel0} abcd = 1,\\
\label{eqn:rel1} (a^{-k}b^k)^{\ell} (c^{-k}b^k)^{\ell-1}(c^{-k}b^{k-1}) = 1, \\
\label{eqn:rel2} (b^{-k}c^k)^{\ell}(d^{-k}c^k)^{\ell-1}(d^{-k}c^{k-1}) = 1, \\
\label{eqn:rel3} (c^{-k}d^k)^{\ell}(a^{-k}d^k)^{\ell-1}(a^{-k}d^{k-1}) = 1, \\
\label{eqn:rel4} (d^{-k}a^k)^{\ell}(b^{-k}a^k)^{\ell-1}(b^{-k}a^{k-1}) = 1. 
\end{align}

The signs of the exponents of the occurences of the generators in the relators \eqref{eqn:rel1}, \eqref{eqn:rel2}, \eqref{eqn:rel3}, and \eqref{eqn:rel4}, respectively, are
\[
(-,+,-,\circ), (\circ,-,+,-), (-,\circ,-,+), \text{ and } (+,-,\circ,-), 
\]
where the coordinates correspond to $(a,b,c,d)$, and $\circ$ indicates that the corresponding generator does not appear.  Also, the relator \eqref{eqn:rel0} has exponent signs $(+,+,+,+)$.      

Now suppose for contradiction that $\pi_1(\Sigma_4(K_{[2\ell,-2k]}))$ is left-orderable.  The above observations provide a number of restrictions on the purported order.  For instance, we cannot have $a < 1$, $b > 1$, and $c < 1$, since this implies the left hand side of \eqref{eqn:rel1} is positive, and consequently not 1.  By analyzing all of these conditions, we may assume that after possibly applying a cyclic automorphism to $(a,b,c,d)$, we have $a,b > 1$ and $c,d <1$.  

Since $b > 1$ and $c^{-1} > 1$, relation \eqref{eqn:rel1} gives 
\begin{equation}\label{eqn:ab}
a^{-k}b^k < 1.
\end{equation}
We rewrite \eqref{eqn:rel4} as
\[
d^{-k}(a^k d^{-k})^{\ell-1}(a^k b^{-k})^{\ell} a^{k-1} = 1.
\]
Since $a > 1$ and $d^{-1}>1$, we get
\begin{equation}\label{eqn:ba}
a^k b^{-k} < 1.
\end{equation}
We now claim that 
\begin{equation}\label{eqn:da}
da>1.
\end{equation}
If not, then $da<1$, and hence, since $d < 1$, we have $a<d^{-1} \leq d^{-k}$.  Therefore, $a^{-1}d^{-k} > 1$.  Again, we rewrite \eqref{eqn:rel4} as
\begin{equation}\label{eqn:adab}
a^k (d^{-k} a^k)^{l-1}(b^{-k} a^k)^{\ell} (a^{-1} d^{-k}) = 1.  
\end{equation}
Note that $a > 1$, $d^{-1}>1$, and further, by \eqref{eqn:ab}, $b^{-k}a^k >1$.  Therefore, the product of the terms above on the left hand side of \eqref{eqn:adab} is positive.  This gives a contradiction.  Therefore, we have that $da > 1$.  

Since $da>1$, \eqref{eqn:rel0} gives 
\begin{equation}\label{eqn:bc}
bc < 1.
\end{equation}
We rewrite \eqref{eqn:rel1} as 
\[
(b^k a^{-k})^{\ell-1} b^k (c^{-k} b^k)^{\ell-1} c^{-k} b^{k-1} a^{-k} = 1.  
\]
By \eqref{eqn:ba}, we have that $a^{-k} > b^{-k}$, and so we get
\[
1 > (b^k a^{-k})^{\ell-1} b^k (c^{-k}b^k)^{\ell-1} c^{-(k-1)} (c^{-1}b^{-1}).  
\]
However, since $b > 1, c^{-1} > 1$, $b^k a^{-k} > 1$ by \eqref{eqn:ba},  and $c^{-1}b^{-1} > 1$ by \eqref{eqn:bc}, we have obtained a  contradiction.  
\end{proof}

If $K$ is a two-bridge knot corresponding to a rational number of the form $[2a_1,2a_2,\ldots,2a_k]$ where $a_i > 0$, for $1 \leq i \leq k$, then $\Sigma_n(K)$ is the branched double cover of an alternating link \cite{MV}.  Therefore, for any $n \geq 2$, $\Sigma_n(K)$ is an L-space by \cite[Proposition 3.3]{OSBranched} and $\pi_1(\Sigma_n(K))$ is not left-orderable \cite[Theorem 4]{BGW} (see also \cite{Greene2011, Ito2013, LL}).  The results of \cite{Hu2013,Tran2013} say that for certain two-bridge knots $K$, we have that $\pi_1(\Sigma_n(K))$ is left-orderable for all sufficiently large $n$.  The situation for cyclic branched covers of torus knots is described in Theorem~\ref{thm:torus}.  These results all suggest the following question (compare to Conjecture~\ref{conj:asymptotics}).  

\begin{question}\label{quest:largerbranchedlo}
Let $K$ be a knot in $S^3$.  If $\pi_1(\Sigma_m(K))$ is left-orderable, is $\pi_1(\Sigma_n(K))$ left-orderable for $n \geq m$?  
\end{question}

Recall that the answer to Question~\ref{quest:largerbranchedlo} is yes if $K$ is prime and $m$ divides $n$.  An interesting example from this point of view is $5_2 = K_{[4,-2]}$.  It is shown in \cite{Hu2013} (see also \cite{Tran2013}) that $\pi_1(\Sigma_n(5_2))$ is left-orderable for all $n \geq 9$.  On the other hand, $\pi_1(\Sigma_n(5_2))$ is not left-orderable for $n = 2$ (since $\Sigma_2(5_2)$ is a lens space), $n = 3$ \cite{DPT2005}, and $n = 4$, by Theorem~\ref{thm:twobridge}.  

\begin{question}
Is $\pi_1(\Sigma_n(5_2))$ left-orderable for $5 \leq n \leq 8$?  
\end{question}

Masakazu Teragaito has informed us that Mitsunori Hori has shown that $\Sigma_5(5_2)$ is an L-space, and thus the conjectural equivalence of \eqref{equiv:lo} and \eqref{equiv:lspace} would suggest that $\pi_1(\Sigma_5(5_2))$ is not left-orderable.  

\section{Three-strand pretzel knots}\label{sec:pretzels}
\begin{proof}[Proof of Theorem~\ref{thm:not left}]
Let $K = P (2k+1,2\ell+1,2m+1)$ where $k, \ell, m \geq 1$.  That $\Sigma_3(K)$ is an L-space is proved in \cite{Ter}.  We will construct an explicit presentation of $\pi_1(\Sigma_3(K))$.  Let $M$ be the exterior of $K$.  It is shown in \cite{Tr} that $\pi_1 (M)$ has a presentation with generators $x,y,z$ and relators
\begin{equation}\label{eq:not left1}
(xy^{-1})^m x(xy^{-1})^{-m} (yz^{-1})^{k+1} z^{-1} (yz^{-1})^{-(k+1)} \end{equation}
\begin{equation}\label{eq:not left2}
(yz^{-1})^k y(yz^{-1})^{-k} (zx^{-1})^{\ell+1} x^{-1} (zx^{-1})^{-(\ell+1)} 
\end{equation}
\begin{equation}\label{eq:not left3}
(zx^{-1})^\ell z(zx^{-1})^{-\ell} (xy^{-1})^{m+1} y^{-1} (xy^{-1})^{-(m+1)}
\end{equation}
where $x,y,z$ are meridians of $K$.

It is straightforward to verify that the product of the relators \eqref{eq:not left1}, \eqref{eq:not left2} and \eqref{eq:not left3} is the identity in the free group on $x,y$ and $z$,  and so relator \eqref{eq:not left3} may be eliminated.  Let $X$ be the 2-complex corresponding to the resulting presentation; thus $X$ has a single 0-cell $c$, three 1-cells $e_x, e_y, e_z$ corresponding to $x,y,z$, respectively, and two 2-cells $D_1,D_2$ corresponding to the relators \eqref{eq:not left1} and \eqref{eq:not left2}.  Let $p : X_3 \to X$ be the 3-fold cyclic cover.  Then $p^{-1}(c) = \{c_0, c_1,c_2\}$, say, and $p^{-1} (e_x) = e_x^{(0)}\cup e_x^{(1)} \cup e_x^{(2)}$, where $e_x^{(i)}$ is a path in $X_3$ from $c_i$ to $c_{i+1}$, $i\in \zed/3$; similarly for $p^{-1} (e_y)$ and $p^{-1} (e_z)$.  Each 2-cell $D_j$, $j\in \{1,2\}$, lifts to three 2-cells $D_j^{(i)}$, $i\in \zed/3$. 

To get a presentation for $\pi_1(X_3)$ we choose a maximal tree in the 1-skeleton $X_3^{(1)}$; we take this to be $e_z^{(0)}\cup e_z^{(1)}$. 
The path $e_x^{(i)}$ now represents an element $x_i \in \pi_1 (X_3)$, and similarly for $e_y^{(i)}, e_z^{(i)}$. 
Note that $z_0 = z_1 =1 \in \pi_1 (X_3)$. 
The 2-cells $D_j^{(i)}$ give the following relations, where $i\in \zed/3$: 
\begin{equation}\label{eqn:pretz1}
(x_i y_i^{-1})^m x_i (x_{i+1} y_{i+1}^{-1} )^{-m} (y_{i+1} z_{i+1}^{-1})^{k+1} 
z_i^{-1}(y_iz_i^{-1})^{-(k+1)} =1
\end{equation}
\begin{equation}\label{eqn:pretz2}
(y_i z_i^{-1})^k y_i (y_{i+1} z_{i+1}^{-1})^{-k} (z_{i+1} x_{i+1}^{-1})^{\ell+1} 
x_i ^{-1} (z_i x_i^{-1})^{-(\ell+1)} =1.
\end{equation}
Note that $\pi_1(X_3) \cong \pi_1 (M_3)$, where $M_3$ is the three-fold cyclic cover of $M$.  Let $\mu$ be a meridian of $K$ and $\pi:M_3 \to M$ the covering map.  Then, recall that $\Sigma_3(K)$ is given by $M_3(\pi^{-1}(\mu^3))$.   Thus, to get a presentation of $\pi_1 (\Sigma_3 (K))$ we must adjoin the branching relations $x_0 x_1 x_2 = y_0 y_1 y_2 = z_0 z_1 z_2 =1$.  Since we have that $z_0 = z_1 = 1$, we must have that $z_2 = 1$ as well.  Eliminating $z_0, z_1,$ and $z_2$ from \eqref{eqn:pretz1} and \eqref{eqn:pretz2}, we obtain
\begin{equation}\label{eqn:pretz1new}
(x_i y_i^{-1})^m x_i (x_{i+1} y_{i+1}^{-1})^{-m} y_{i+1}^{k+1} y_i ^{- (k+1)} =1 
\end{equation}
\begin{equation}\label{eqn:pretz2new}
y_i^{k+1} y_{i+1}^{-k} x_{i+1}^{-(\ell+1)} x_i^\ell =1. 
\end{equation}

As in the proof of Theorem~\ref{thm:twobridge}, the relations \eqref{eqn:pretz1new} and \eqref{eqn:pretz2new} have the property that for each generator, the exponents of all the occurrences of that generator in the relator have the same sign.  These signs are given in Tables~\ref{table:pretzrel1} and \ref{table:pretzrel2}, for relations \eqref{eqn:pretz1new} and \eqref{eqn:pretz2new} respectively.

\begin{table}
\begin{center}
\begin{tabular}{| c | c | c | c | c | c | c |}
\hline
$i$  & $x_0$ & $x_1$ & $x_2$ & $y_0$ & $y_1$ & $y_2$ \\
\hline
$0$ & + & $\--$ & $\circ$ & $\--$ & + & $\circ$ \\
\hline
1 & $\circ$ & + & $\--$ & $\circ$ & $\--$ & + \\
\hline
2 & $\--$ & $\circ$ & + & + & $\circ$ & $\--$ \\
\hline
\end{tabular}
\vspace{.1in}
\caption{Signs of the exponents of the generators as they appear in the relation \eqref{eqn:pretz1new} for $i \in \mathbb{Z}/3$.  The notation $\circ$ indicates that the generator does not appear in the corresponding relation.}\label{table:pretzrel1}
\end{center}
\end{table}

\begin{table}
\begin{center}
\begin{tabular}{| c | c | c | c | c | c | c |}
\hline
$i$  & $x_0$ & $x_1$ & $x_2$ & $y_0$ & $y_1$ & $y_2$ \\
\hline
$0$ & + & $\--$ & $\circ$ & $+$ & $\--$ & $\circ$ \\
\hline
1 & $\circ$ & + & $\--$ & $\circ$ & $+$ & $\--$ \\
\hline
2 & $\--$ & $\circ$ & + & $\--$ & $\circ$ & $+$ \\
\hline
\end{tabular}
\vspace{.1in}
\caption{Signs of the exponents of the generators as they appear in the relation \eqref{eqn:pretz2new} for $i \in \mathbb{Z}/3$. The notation $\circ$ indicates that the generator does not appear in the corresponding relation.}\label{table:pretzrel2}
\end{center}
\end{table}
Suppose that there exists a left-invariant order, $<$, on $\pi_1(\Sigma_3 (K))$.  Since there is an automorphism of $\pi_1 (\Sigma_3 (K))$ sending $x_i$ to $x_{i+1}$, if some $x_i =1$ then $x_0 = x_1 =x_2 =1$.  Relation \eqref{eqn:pretz2new}
 then gives $y_i^{k+1} = y_{i+1}^k$, $i\in \zed/3$, which implies that each $y_i$ has finite order. Since we are assuming that $\pi_1(\Sigma_3 (K))$ is left-orderable, $y_0 = y_1 = y_2 =1$ and so $\pi_1 (\Sigma_3(K)) =1$, a contradiction.  Similarly, it follows that no $y_i$ is trivial in $\pi_1 (\Sigma_3 (K))$.

We say that an element $g \in \pi_1 (\Sigma_3 (K))$ is {\em positive} (respectively {\em negative}) if $g>1$ (respectively $g<1$).  Note that a product of elements with the same sign cannot be the identity.  Applying this to the relations \eqref{eqn:pretz1new} and \eqref{eqn:pretz2new}, Tables~\ref{table:pretzrel1} and \ref{table:pretzrel2} show that, for $i\in \zed/3$, if $x_i$ and $x_{i+1}$ have opposite sign then $y_i$ and $y_{i+1}$ have the same sign.  On the other hand, the relations $x_0 x_1 x_2 =1$, $y_0 y_1 y_2 =1$ show that the $x_i$ cannot all have the same sign, and the $y_i$ cannot all have the same sign.  This is clearly a contradiction. 
\end{proof}

\section{More examples}\label{sec:misc}
We discuss a few more families of cyclic branched covers of knots.

\subsection{Knot epimorphisms}
Other examples of knots $K$ with $\pi_1(\Sigma_n(K))$ left-orderable arise from the following easy lemma.

\begin{lemma}\label{lem:knotepi}
Let $K, K'$ be prime knots in $S^3$.  Suppose there exists an epimorphism $\varphi: \pi_1(S^3 - K) \to \pi_1(S^3 - K')$ such that $\varphi(\mu) = \mu'$ for meridians $\mu, \mu'$ of $K, K'$ respectively.  If $\pi_1(\Sigma_n(K'))$ is left-orderable, then so is $\pi_1(\Sigma_n(K))$.  
\end{lemma}
\begin{proof}
Since there is an exact sequence
\[
1 \to \pi_1(\Sigma_n(K)) \to \pi_1(S^3 - K)/\langle \langle \mu^n \rangle \rangle \to \Z/n \to 1, 
\]
and similarly for $K'$, we have that $\varphi$ induces an epimorphism from $\pi_1(\Sigma_n(K))$ to $\pi_1(\Sigma_n(K'))$.  The result now follows from Theorem~\ref{thm:brw}.  
\end{proof}

Examples to which Lemma~\ref{lem:knotepi} applies are given in \cite{KS1,KS2}, which list all pairs of distinct prime knots $K, K'$ with at most 10 crossings such that there exists an epimorphism $\pi_1(S^3 - K) \to \pi_1(S^3 - K')$ (written $K \geq K'$ in the notation of \cite{KS1,KS2}).  The authors note in \cite{KS2} that in all cases, there is actually a meridian-preserving epimorphism.  Also, the knot $K'$ is either $3_1$, $4_1$, or $5_2$.  It follows from Theorem~\ref{thm:torus} that for the knots with $K \geq 3_1$, we have $\pi_1(\Sigma_n(K))$ is left-orderable for $n \geq 6$.  It also follows from \cite{Hu2013} (see also \cite{Tran2013}) that for the knots $K$ with $K \geq 5_2$, we have $\pi_1(\Sigma_n(K))$ is left-orderable for $n \geq 9$.  Since all cyclic branched covers of $4_1$ have non-left-orderable fundamental group \cite{DPT2005}, we are not able to say anything about $\pi_1(\Sigma_n(K))$ for $K \geq 4_1$.

\subsection{Unknotting number one, determinant one knots}

\begin{proposition}
Let $K$ be an unknotting number one knot with $det(K) = 1$.  If the Li-Roberts Conjecture is true, then $\Sigma_2(K)$ is an excellent manifold.
\end{proposition}
\begin{proof}
Suppose that $K$ is an unknotting number one knot with $det(K) = 1$.  Since $K$ has unknotting number 1, it follows that $\Sigma_2(K)$ is given by a half-integral surgery on a knot $K'$ in $S^3$.  Since $det(K) = 1$, $H_1(\Sigma_2(K)) = 0$ and thus $\Sigma_2(K) = S^3_{\pm \frac{1}{2}}(K')$.  The result now follows from Conjecture~\ref{conj:lr} and Lemma~\ref{lem:toroidalzs3}. 
\end{proof}

\bibliography{biblio}
\bibliographystyle{amsalpha}

\end{document}